\documentclass[11pt]{article}
\usepackage[utf8]{inputenc}
\usepackage[T1]{fontenc}
\usepackage[left = 3cm, right = 3cm, top=3cm, bottom=3cm]{geometry}
\usepackage{setspace}
\usepackage{cellspace}
\usepackage[x11names]{xcolor}
\usepackage{mathtools,amssymb,mathrsfs}
\usepackage{amsmath}
\usepackage{stmaryrd}
\usepackage{cellspace}
\usepackage{calc, ifthen, xspace}
\usepackage{tikz,tkz-tab}
\usepackage{fancybox}
\usepackage{amsthm}
\usepackage{enumitem}
\usepackage{empheq}
\usepackage{graphicx}
\usepackage{dsfont}
\usepackage{mdwlist}
\usepackage[numbers]{natbib}
\usepackage{hyperref}

\author{Corentin Faipeur\footnote{ENS Lyon -- corentin.faipeur@ens-lyon.fr}}

\title{Glauber dynamics and coupling-from-the-past for Gaussian fields}

\newtheorem{thm}{Theorem}[section]
\newtheorem{prop}[thm]{Proposition}
\newtheorem{defi}[thm]{Definition}

\newtheorem{lemme}[thm]{Lemma}

\newtheorem{rem}[thm]{Remark}

\newcommand{\E}{\mathbf{E}}
\renewcommand{\P}{\mathbf{P}}

\newcommand{\R}{\mathbf{R}}
\newcommand{\Q}{\mathbf{Q}}
\newcommand{\Z}{\mathbf{Z}}
\newcommand{\N}{\mathbf{N}}

\renewcommand{\d}{\mathrm{d}}

\newcommand{\ind}[1]{\mathbf{1}_{\{#1\}}}
\newcommand{\ssi}{\Longleftrightarrow}

\newcommand{\et}{\text{ and }}
\newcommand{\si}{\text{ if }}
\newcommand{\otw}{\text{ otherwise }}

\newcommand{\cas}[4]{\begin{cases}
#1 & #2\\
#3 & #4
\end{cases}}

\newcommand{\C}{\mathcal{C}}
\renewcommand{\L}{\mathcal{L}}
\newcommand{\M}{\mathcal{M}}
\newcommand{\Norm}{\mathcal{N}}
\newcommand{\PP}{\mathcal{P}}
\newcommand{\A}{\mathcal{A}}

\renewcommand{\S}{\mathcal{S}}

\newcommand{\as}[1]{\overset{\text{a.s.}}{#1}}

\begin{document}
\maketitle


\begin{abstract}
    We study the representation of stationary Gaussian Markov random fields as factors of i.i.d. processes, with a focus on their approximation by finitely dependent distributions. Our model is a Gaussian field on $\Z^d$ such that the conditional law of the field at any site is Gaussian of mean $\varepsilon$ times the average of its neighbours, and of variance 1.
    Building on coupling-from-the-past (CFTP) techniques, we prove that for sufficiently small $\varepsilon$, the distribution of the field can be written as an explicit factor of an i.i.d. process. Furthermore, we construct approximations by finitely dependent fields that are close in total variation to the original field, with exponential decay when the allowed range of dependence grows. 
    We first do the proof for a truncated version of this Gaussian model, showing in this case that the associated field admits a finitary coding with exponential tails, providing a new application of high-noise condition of \cite{HaggstromSteif} for an uncountable state space.
    The proof for the original model is more intricate. Our approach extends classical CFTP-based constructions by developing a stratified coupling method tailored to the continuous and unbounded nature of the Gaussian setting.
\end{abstract}

\section{Introduction and main results}
This paper is concerned with the existence of \emph{factor maps} from i.i.d. (independent and identically distributed) processes to stationary Markov random fields (precise definitions will be given below). For simplicity, we consider random fields on $\Z^d$ for some $d\geq 1$, although the theory can be naturally extended to more general graphs.
Originally, the interest for factor maps came from ergodic theory.
Processes of special interest in ergodic theory are the so-called \emph{Bernoulli shifts}, that are the random fields (initially in dimension $d=1$) which are isomorphic to an i.i.d. process, meaning that there is an almost everywhere invertible factor map (see section \ref{Factors of product measures} for a precise definition) between the two.
Ornstein proved in \cite{ornstein} the isomorphism theorem for Bernoulli shifts, which states that two i.i.d. processes on $\Z$ of equal entropy are in fact isomorphic; he further showed that factors of i.i.d. processes are Bernoulli shifts. These results extend to higher dimensions (see \cite{KatznelsonWeiss} and \cite{OrnsteinWeiss}).

\subsection{Markov random fields and Gibbs measures}\label{Markov random fields and Gibbs measures}
A random field with state space $S$ is a collection $X=(X_i)_{i \in \Z^d}$ of $S-$valued random variables indexed by $\Z^d$. Equivalently, the distribution of $X$ is a probability measure on $\Omega:=S^{\Z^d}$. In this paper, $S$ can be either a finite set, equipped with the discrete $\sigma-$algebra, or a subset of $\R$, equipped with the Borel $\sigma-$algebra. In both cases, $\Omega$ is equipped with the product $\sigma-$algebra. Elements of $\Omega$ are often called \emph{configurations}.
We focus on distributions on $\Omega$ that are invariant by the translations of $\Z^d$. Random fields whose distribution is translation-invariant are called \emph{stationary}.
We are more specifically interested in stationary random fields that moreover satisfy the \emph{local Markov property}: it says that the conditional distribution of the field at a given site $i$ of $\Z^d$ only depends on the values of the field on some finite neighbourhood around $i$.
\begin{defi}
    A random field $X$ is said to be a \emph{Markov random field} with boundary set $B\Subset \Z^d \setminus\{0\}$ (i.e. $B$ is a finite subset of $\Z^d \setminus\{0\}$) if its distribution $\mu$ satisfies for all $i \in \Z^d$,
    \begin{equation}\label{Markov_property}
        \mu(X_i \in \cdot \mid (X_j)_{j \neq i}) =\mu(X_i \in \cdot \mid (X_j)_{j \in i+B}) \quad \mu-\text{a.s.}
    \end{equation}
\end{defi}
When a Markov random field with boundary set $B$ is fixed, we will write $j \sim i$ when $j \in i+B$; since the Markov property remains true by taking a larger boundary set, we can always assume that $B$ is symmetric, so that $j \sim i$ is equivalent to $i\sim j$.

Stationary Markov random fields often arise as statistical mechanics models.
Below we briefly introduce the notions of specification and Gibbs measure, and refer to \cite{friedli_velenik_2017} and \cite{georgii} for a large overview of the theory.
A statistical mechanics model (on $\Z^d$) is defined through a potential $\Phi$, which is a collection
of functions $\Phi_\Delta:\Omega \to \R$, for every $\Delta \Subset \Z^d$. The energy of a configuration is represented by the \emph{Hamiltonian} of the model, defined for every $x\in \Omega$ and $\Lambda \Subset \Z^d$ by
$$H^\Phi_\Lambda(x)=\sum_{\Delta \cap \Lambda \neq \emptyset} \Phi_\Delta(x).$$
To ensure that the above sum is convergent, we assume that the potential satisfies $\Phi_\Delta \equiv 0$ for all $\Delta$ with diameter greater that a certain integer $r$; such potentials are said to have finite range.
Since we aim to define stationary random fields, we furthermore assume that the potential is translation-invariant, in the sense that for every $\Delta \Subset \Z^d$, $i \in \Z^d$ and $x\in \Omega$, one has $\Phi_{i+\Delta}(x)=\Phi_\Delta((x_{i+j})_{j\in \Z^d})$.
We endow the state space $S$ with an a priori measure $\lambda$, which is taken to be the counting measure if $S$ is countable and the Lebesgue measure if $S$ is an uncountable subset of $\R$. The associated product measure over $\Lambda \subset \Z^d$ is denoted $\lambda^\Lambda$.
The \emph{specification} of the model is a consistent set of conditional distributions $\pi^\Phi=(\pi^\Phi_\Lambda)_{\Lambda \Subset \Z^d}$: for all $\Lambda \Subset \Z^d$ and $\xi \in \Omega$, $\pi^\Phi_\Lambda(\cdot \mid \xi)$ is the probability measure on $S^\Lambda$ defined by
$$\pi^\Phi_\Lambda(\cdot \mid \xi):= \dfrac{1}{Z^\xi_\Lambda} \int_{S^\Lambda} \ind{x_\Lambda \in \cdot} \exp(-H^\Phi_\Lambda(x_\Lambda \xi_{\Lambda^c})) \ \d\lambda^\Lambda(x_\Lambda)$$
where $Z^\xi_\Lambda:=\int_{S^\Lambda}\exp(-H^\Phi_\Lambda(x_\Lambda \xi_{\Lambda^c}))\d\lambda^\Lambda(x_\Lambda)$ is just a renormalising constant, and $x_\Lambda \xi_{\Lambda^c}$ is the configuration that agrees with $x_\Lambda$ on $\Lambda$ but with $\xi$ on $\Lambda^c$.
A probability measure on $\Omega$ having the conditional distributions prescribed by $\pi^\Phi$ is called a \emph{Gibbs measure} for $\pi^\Phi$ (or equivalently for the potential $\Phi$).
In other words, $\mu$ is a Gibbs measure for $\pi^\Phi$ if for all $\Lambda \Subset \Z^d$,
\begin{equation*}\tag{DLR}\label{DLR}
    \mu\big( X|_\Lambda\in \cdot \mid \forall i \notin \Lambda, X_i = \xi_i\big) = \pi^\Phi_\Lambda(\cdot \mid \xi) \text{ for $\mu-$a.e } \xi \in \Omega,
\end{equation*}
where $X|_\Lambda:=(X_i)_{i \in \Lambda}$ is the restriction of $X$ to $\Lambda$.
The above relation is called the DLR equation, for Dobrushin, Landford and Ruelle, who introduced this formalism for the definition of statistical mechanics models, see \cite{Dobrushin} \cite{LanfordRuelle}.
When $S$ is finite or compact, existence of a Gibbs measure for a given potential follows from compactness of $\Omega$ (see e.g. Theorem 6.26 in \cite{friedli_velenik_2017}).
Uniqueness, however, is a more intricate question, which is central in statistical mechanics.

By the finite-range assumption, a random field whose distribution is a Gibbs measure is a Markov random field. In fact, it satisfies the local Markov property \eqref{Markov_property} by taking as boundary set the reunion of subsets $\Delta \Subset \Z^d$ that contain 0 and such that $\Phi_{\Delta} \not \equiv 0$, minus 0.
Conversely, it follows from the Hammersley-Clifford theorem \cite{HammCliff}, \cite{Averintsev}, \cite{Spitzer}, that the distribution $\mu$ of a stationary Markov random field can always be expressed as a Gibbs measure for some (finite range and translation-invariant) potential when the state space is finite and $\mu$ is fully supported.

A well-known example of stationary Markov random field, which is a fundamental model in statistical mechanics, is the \emph{Ising model}.
Fix $S=\{-1,1\}$ and let $\beta \in \R$. The potential of the Ising model on $\Z^d$ has the following simple form:
$$\forall \Delta \Subset \Z^d, \forall \omega \in \{\pm1\}^{\Z^d}, \ \Phi_\Delta(\omega)=\begin{cases}
    -\beta \omega_i \omega_j & \si \Delta=\{i,j\}, i \sim j \\
    0 & \otw
\end{cases}.$$
Its Hamiltonian is therefore given, for every $\Lambda \Subset \Z^d$ and $\omega \in \{\pm1\}^{\Z^d}$, by 
$$H_\Lambda(\omega) = - \beta \sum_{i \sim j \atop \{i,j\} \cap \Lambda \neq \emptyset} \omega_i \omega_j.$$
It is well known that for $d\geq 2$, there exists a critical parameter $\beta_c(d) \in (0, \infty)$ such that for all $\beta \in [0, \beta_c(d))$, there is a unique Gibbs measure $\mu_\beta$ satisfying the \eqref{DLR} equation, whereas for $\beta > \beta_c(d)$, several Gibbs measures coexist.
In fact, it can be shown that if one considers an increasing sequence of boxes $\Lambda_n \nearrow \Z^d$ and the measures $\mu_{\beta, \Lambda_n}^+:=\pi^{\Phi}_{\Lambda_n}(\cdot \mid +)$ where $+$ denotes the configuration with $+1$ everywhere, then for all $\beta\geq0$ this sequence of measures converges (for the weak convergence topology) towards a Gibbs measure $\mu_\beta^+$ (independent of the choice of the $\Lambda_n$). A Gibbs measure $\mu_\beta^-$ is obtained similarly.
One can then prove that the coexistence of several Gibbs measures is equivalent to $\mu_\beta^- \neq\mu_\beta^+$, and that $\mu_\beta^-$ and $\mu_\beta^+$ indeed differ if $\beta$ is sufficiently large.

\subsection{Factors of product measures}\label{Factors of product measures}
Let $\mu$ be a translation-invariant distribution on $S^{\Z^d}$ (Markovian assumption is not required for the following general definitions).
We say that $\mu$ is a \emph{factor of a product measure} or \emph{factor of an i.i.d. process} (abbreviated FIID in the following)
if there exists a probability space $(\mathcal{Y}, \mathcal{F}, \nu)$ and a measurable map $F : \mathcal{Y}^{\Z^d} \to S^{\Z^d}$ which commutes with all the translations of $\Z^d$ such that $\mu$ is the image of $\nu^{\otimes \Z^d}$ by $F$.
It is equivalent to say that there exists an i.i.d. $\mathcal{Y}-$valued process $Y= (Y_i)_{i \in \Z^d}$ such that $F(Y)$ has distribution $\mu$. 
The factor map $F$ is called a \emph{coding} from $\nu^{\otimes \Z^d}$ to $\mu$.

From statistical mechanics point of view, the notion of FIID is not satisfying since it does not catch phase transitions.
Indeed, it is known that the ``plus state'' of the Ising model, that is the measure $\mu_\beta^+$ defined in Section \ref{Markov random fields and Gibbs measures}, is an FIID for all $\beta \geq0$ (proved first by Ornstein and Weiss in an unpublished manuscript, then by Adams \cite{Adams}).
Van den Berg and Steif \cite{VdB_steif} then provide a way of detecting phase transitions through the notion of \emph{finitary} codings.
The \emph{coding radius} of the factor map $F$ is defined as the random radius around the origin that one needs to reveal in the i.i.d. process to determine the spin at 0.
More precisely, define the coding radius of $F$ at site $i \in \Z^d$ and configuration $y \in \mathcal{Y}^{\Z^d}$ as 
$$R_i(y) := \inf \{r \geq 0:\forall y' \in \mathcal{Y}^{\Z^d}, y'|_{B^i_r}=y|_{B^i_r} \implies  F(y)_i = F(y')_i\}$$
where $B^i_r := \{j \in \Z^d : ||j-i||_1 \leq r\}$ is the ball of radius $r$ around $i$ in $\Z^d$, and $y|_{B^i_r}$ is the restriction of $y$ to $B^i_r$.
When $Y$ is as above, the $R_i(Y)$ for $i \in \Z^d$ are random variables that are equally distributed, due to translation invariance.
Set $R:=R_0(Y)$ (the choice of 0 is made arbitrarily, by translation invariance of $\mu$); $R$ is called the coding radius of $F$.
If $R$ is almost surely finite, the coding is said to be \emph{finitary}. In this case, we say that $\mu$ is a \emph{finitary factor} of an i.i.d. process, abbreviated FFIID.
Moreover, when the coding radius satisfies $\P(R\geq n) \leq Ce^{-c n}$ for some $c,C>0$ and all $n \in \N$, we say that $\mu$ is an FFIID with \emph{exponential tails}.
The main result of \cite{VdB_steif} is that the presence of a phase transition is ``a fundamental obstruction for finitary coding''; in words, if there exists a distribution different from $\mu$ with the same specification, this prevents $\mu$ from being an FFIID.
This shows that the plus state of the Ising model on $\Z^d$ is \emph{not} FFIID if $\beta> \beta_c(d)$, despite being a Bernoulli shift.
It is also proved in \cite{VdB_steif} that in the regime of uniqueness, the Gibbs measure of the Ising model is an FFIID, with exponential tails if $\beta<\beta_c$, but with coding radius's $d-$th moment necessarily infinite at $\beta=\beta_c$ (this last result was jointly proved with Peres in the same paper).
Note that the existence of a finitary coding at the critical point follows from the continuity of the phase transition for the Ising model (\cite{Yang_2dIsing} \cite{AizenmanFernandez_cont_Ising} \cite{ADS_cont_Ising}).

The work of van den Berg and Steif makes use of coupling-from-the-past (CFTP) arguments, developed by Propp and Wilson \cite{propp_wilson} to build their famous exact sampling algorithm.
In particular, to construct a finitary coding for the Ising model up to the critical point, they implement a monotone coupling, which is possible only if $\beta \geq0$.
Another construction due to Häggström and Steif \cite{HaggstromSteif}, also relying on CFTP arguments, enables to get a finitary coding for the Ising model with $\beta$ close to 0, but possibly negative.
Although not specified in their paper, the coding has exponential tails (this is because in \cite{HaggstromSteif}, as in \cite{VdB_steif}, the authors are mostly interested in codings from finite-valued i.i.d. processes; this additional restriction makes the tail of the coding radius harder to control).
In fact, they prove the existence of a finitary coding for all stationary Markov random fields with finite state space which satisfy a \emph{high-noise} condition. This is a condition on specifications, which, with the notations of Section \ref{Markov random fields and Gibbs measures}, amounts to
\begin{equation}\label{eq:high_noise}
    \sum_{s \in S} \inf_{\xi \in \Omega} \pi_{\{i\}}^{\Phi}(\{s\} \mid \xi)>1-\frac{1}{2d}, \ i\in \Z^d.
\end{equation}
The left-hand side is called the \emph{multigamma admissibility} in \cite{HaggstromSteif}, following the terminology introduced in \cite{MurdochGreen} (it does not depend on $i$, by translation-invariance).
Note that \eqref{eq:high_noise} immediately implies the uniqueness condition obtained by van den Berg and Maes in \cite{VdBMaes}, so that we are necessarily in the regime of uniqueness of the Gibbs measure when the high-noise condition is fulfilled.
For the Ising model, it corresponds to taking $|\beta| < \frac{1}{4d}\log\left(1+\frac{2}{2d-1}\right)$.
In \cite{spinka10}, Spinka improved the result of \cite{HaggstromSteif}, proving that stationary Markov random fields (with finite state space) that satisfy exponential strong spatial mixing are FFIID with exponential tails, and this is a weaker condition than \eqref{eq:high_noise}.

We end this paragraph by making the connection between FFIID and finitely dependent processes.
For any integer $l \in \N$, we say that a measure $\mu$ on $\Omega$ has finite dependence of range $l$ (or $\mu$ is $l-$dependent) if for any $A,B \subset \Z^d$ such that $d(A,B) > l$, the restrictions $X|_A$ and $X|_B$ are independent, when $X$ has law $\mu$; $\mu$ is \emph{finitely dependent} if it is $l-$dependent for some $l \in \N$.
If $\mu$ is an FFIID with coding radius almost surely bounded by some constant, it is said to be a \emph{block factor} of an i.i.d. process. In this case, it is trivially finitely dependent (take $l$ equal to twice the almost sure bound on the coding radius).
More generally, the tail of the coding radius controls the distance (in total variation) of the field to finitely dependent processes.
Recall that the \emph{total-variation distance} between two probability measures $\mu$ and $\nu$ on a measurable space $(\Omega, \mathcal{F})$ is defined by $d_{\mathrm{TV}}(\mu, \nu) = \sup\{|\mu(A) - \nu(A)|, A\in \mathcal{F}\}$.
It is well known that we have $d_{\mathrm{TV}}(\mu, \nu) = 1- \sup \P(X = Y)$ where the supremum is taken over all coupling $(X,Y)$ of $\mu$ and $\nu$.
If $\mu$ is an FIID with coding radius $R$, for every $l\in \N$, there exists an $l-$dependent distribution $\nu_l$ on $\Omega$ such that for all $\Lambda \Subset \Z^d$,
\begin{equation}\label{eq:distance_to_l-dep}
    d_{\mathrm{TV}}(\mu|_\Lambda, \nu_l |_\Lambda) \leq |\Lambda| \P(2R > l).
\end{equation}
This is obtained by truncating the factor map. We give a proof of this fact below (see Proposition \ref{prop:radius_controls_distance_l-dep}), but we do not claim that this is a new result.
Conversely, Spinka proved in \cite{spinka13} that finitely dependent distributions are FFIID, when the state space is finite (or countable).


\subsection{Gaussian models and main results}
A lot of results regarding finitary codings for Markov random fields have been obtained when the state space $S$ is finite ---see \cite{spinka10} for a comprehensive review of the literature, and for results on various statistical mechanics models.
In this paper, our goal is to study the case when $S$ is an uncountable subset of $\R$. For concreteness, we consider an explicit Gaussian model and a truncated version, although the very same proof would carry over to more general setups under relatively mild assumptions on the marginals and the specification. 
For example, still in the Gaussian case, it applies directly to any finite-range Hamiltonian with small enough off-diagonal terms.
Also, things should work in a similar way for the corresponding integer-valued height models with the same Hamiltonians.
It is unclear whether there is a reasonable way to describe the minimal required assumptions, but the interested reader can have a look at the beginning of Section \ref{Proofs for the Gaussian models} to see which properties are needed.

\paragraph{Gaussian model.}
Fix $S=\R$ (thus $\Omega=\R^{\Z^d}$) and a parameter $\varepsilon \in (-1,1)$. For $m \in \R$, denote by $\Norm(m, 1)$ the normal distribution of mean $m$ and variance 1.
We consider a centered stationary Gaussian field $(X_i)_{i \in \Z^d}$ on $\Z^d$ such that the conditional distribution of $X_0$ given the rest of the configuration is Gaussian of variance $1$ and mean $\varepsilon \times$ (average of the neighbours). It means that its distribution, denoted $\mu_\varepsilon$, satisfies for all $i \in \Z^d$
\begin{equation}\label{def_mu}
    \mu_\varepsilon(X_i \in \cdot \mid (X_j)_{j\neq i})(\xi) = \Norm\Big(\frac{\varepsilon}{2d}\sum_{j, \Vert j-i\Vert_1=1} \xi_j, 1\Big), \text{ for $\mu_\varepsilon-$a.e. } \xi \in \R^{\Z^d}.
\end{equation}
By definition, $(X_i)_{i \in \Z^d}$ is a Markov random field, with boundary set $B=B^0_1 \setminus\{0\}$ the unit sphere of $\Z^d$.

\begin{rem}\label{rem:def_gaussian_field}
The above description of our model does not guarantee that such a Gaussian field is well defined.
However, assuming that it exists, we can identify from \eqref{def_mu} the coefficients of the inverse of the covariance operator.
If $\Gamma$ is the covariance of $(X_i)_{i\in \Z^d}$, then, given the formula for the density of a Gaussian vector, it follows that we must have $\Gamma^{-1} = I-T$ with $T=\left( \dfrac{\varepsilon}{2d} \ind{i \sim j} \right)_{i,j \in \Z^d}$ and $I$ the identity, when viewed as operators on $\ell^2(\Z^d)$.
The inverse of $I-T$ is well defined when $\varepsilon \in (-1,1)$, because the Neumann series $\sum_{n\geq 0} T^n$ converges in this case.
Indeed, for all $x \in \ell^2(\Z^d)$,
$$|\langle x,Tx\rangle |=\dfrac{|\varepsilon|}{2d} \left|\sum_{i \in \Z^d} \sum_{j,j \sim i} x_ix_j \right| \leq \dfrac{|\varepsilon|}{2d} \sum_{i,j, i \sim j} 2 |x_ix_j| \leq \dfrac{|\varepsilon|}{2d} \sum_{i \sim j} (x_i^2+x_j^2) =|\varepsilon| \sum_{i \in \Z^d} x_i^2,$$
thus 
$$\Vert T \Vert \leq |\varepsilon|<1,$$
where $\langle \cdot,\cdot \rangle$ denotes the usual scalar product on $\ell^2(\Z^d)$ and $\Vert \cdot \Vert$ the associated operator norm.
Then,
$$\langle x,(I-T)x\rangle \geq (1-|\varepsilon|)\, \langle x,x\rangle$$
so we find that $I-T$ is positive definite, so is $\Gamma=(I-T)^{-1}=\sum_{n\geq 0} T^n.$

Therefore, for every $\varepsilon \in (-1,1)$, there exists a unique Gaussian field $(X_i)_{i \in \Z^d}$ with mean 0 and covariance 
$$\operatorname{Cov}(X_i, X_j) =\Gamma(i,j)= \sum_{n \geq0} \left(\dfrac{\varepsilon}{2d}\right)^n \times (\text{number of paths of length $n$ from $i$ to $j$}),$$
and we define $\mu_\varepsilon$ as its distribution.
In the other way, it is now clear that \eqref{def_mu} holds for $\mu_\varepsilon$.
\end{rem}

The description of $(X_i)_{i \in \Z^d}$ through its conditional distributions corresponds to the DLR approach in statistical mechanics. The theory of Gaussian fields seen as Gibbs measures has been well developed in \cite{georgii} (Chapter 13). 
In fact, \eqref{def_mu} describes the singleton part of the specification of the model. It is a general fact (\cite{georgii} Theorem 1.33) that this is sufficient to characterise the whole specification when it has a positive density with respect to some a priori measure (which is the case with the Lebesgue measure in \eqref{def_mu}).
The formula of the Gaussian density allows us to identify from \eqref{def_mu} what is the Hamiltonian of the model:
\begin{equation}\label{hamiltonian}
    H_\Lambda(x) = \frac{1}{2} \sum_{i \in \Lambda} x_i^2 - \frac{\varepsilon}{2d} \sum_{i \sim j \atop \{i,j\} \cap \Lambda \neq \emptyset} x_ix_j, \ \forall \Lambda\Subset \Z^d, x\in \Omega.
\end{equation}
Theorems 13.24 and 13.26 in \cite{georgii} characterise the set of Gibbs measures associated with the above Hamiltonian.
It corresponds to the convex hull of the set of Gaussian field distributions of covariance $\Gamma$ and mean $m\in \R^{\Z^d}$ satisfying
$$\forall i \in \Z^d, \ m_i = \dfrac{\varepsilon}{2d} \sum_{j \sim i} m_j,$$
giving rise to uncountably many Gibbs measures. However, notice that $\mu_\varepsilon$ is the only Gibbs measure for this model which is translation-invariant (because $m_i=m$ for all $i$ implies $m=0$).
Up to renormalisation, in the case $\varepsilon >0$, the Hamiltonian defined in \eqref{hamiltonian} can be rewritten as the Hamiltonian of the discrete \emph{Gaussian free field} with mass $m$ satisfying $(1+m^2)^{-1}=\varepsilon$ (see e.g. \cite{friedli_velenik_2017} Chapter 8).
The case $\varepsilon=1$ would therefore corresponds to the massless ($m=0$) Gaussian free field, for which it is well known that no Gibbs measure exists in dimensions $1$ and $2$.

One of the main purposes of this paper is to investigate how close $\mu_\varepsilon$ is from finitely dependent distributions.
To do this, we propose to write $(X_i)_{i \in \Z^d}$ of law $\mu_\varepsilon$ as a factor of an i.i.d. process, and then to construct a finitely dependent random field that is close to $(X_i)_{i \in \Z^d}$ by truncating the factor map.
Since there are infinitely many distributions having the same specification as $\mu_\varepsilon$, there is no hope of getting a \emph{finitary} coding for $\mu_\varepsilon$ (recall the ``fundamental obstruction'' of \cite{VdB_steif}).
Nevertheless, we construct a (non-finitary) coding for $\mu_\varepsilon$, and obtain the same bound on the distance to finitely dependent fields as if we had a coding radius with exponential tails (though it does not directly follow from \eqref{eq:distance_to_l-dep}).
The following theorem is the main result of the paper.
\begin{thm}\label{thm_mu}
    For $|\varepsilon|$ small enough, 
    \begin{enumerate}
        \item $\mu_\varepsilon$ is a factor of an i.i.d. process.
        \item There exists $c>0$ such that for all $l\in \N$, there exists an $l$--dependent distribution $\nu_l$ on $\R^{\Z^d}$ satisfying for every finite subset $\Lambda \subset \Z^d$,
        $$d_{\mathrm{TV}}(\mu_{\varepsilon}|_\Lambda, \nu_l |_\Lambda) \leq |\Lambda| e^{-cl}.$$ 
    \end{enumerate}
\end{thm}

This is analogous to the aforementioned result of \cite{HaggstromSteif} for the ``very high temperature'' Ising model, as we construct the coding without monotonicity argument (thus allowing $\varepsilon<0$) but, on the other hand, we need a kind of ``high-noise'' condition, that is, $|\varepsilon|$ small enough. However, the proof of \cite{HaggstromSteif} cannot be adapted in this context due to the unboundedness of the spins, contrary to the truncated case that we now introduce.

\paragraph{Truncated Gaussian model.}
For any $L>0$, one may consider a stationary random field $(X_i)_{i \in \Z^d}$ with state space $[-L,L]$, such that the distribution of $X_0$ knowing the rest of the configuration is Gaussian of mean $\frac{\varepsilon}{2d}\sum_{j \sim i} X_j$ conditioned to be in $[-L,L]$.
This corresponds to a Gibbs measure associated with the same Hamiltonian \eqref{hamiltonian} but on a different state space.
More precisely, define a specification $\pi^{\varepsilon,L}=(\pi^{\varepsilon, L}_\Lambda)_{\Lambda \Subset \Z^d}$ by
$$\forall\Lambda \Subset \Z^d, \xi \in [-L,L]^{\Z^d}, \ \pi^{\varepsilon, L}_\Lambda( \cdot \mid \xi) :=\int_{[-L,L]^\Lambda} \ind{x_\Lambda \in \cdot} \exp(-H_\Lambda(x_\Lambda \xi_{\Lambda^c})) \d \lambda^\Lambda(x_\Lambda). $$
Existence of a Gibbs measure compatible with $\pi^{\varepsilon,L}$ is guaranteed by compactness of the state space (see e.g. Theorem 6.26 in \cite{friedli_velenik_2017}).
Moreover, for $|\varepsilon|$ small enough ($|\varepsilon|<1/(2L^2)$ suffices), $\pi^{\varepsilon,L}$ satisfies the Dobrushin's condition of weak dependence (see Theorem 6.35 in \cite{friedli_velenik_2017} or Theorem 8.7 in \cite{georgii}) so there is a unique Gibbs measure specified by $\pi^{\varepsilon,L}$; call it $\mu_\varepsilon^L$.
We prove a result similar to Theorem \ref{thm_mu} for $\mu_\varepsilon^L$.

\begin{thm}\label{thm_mu_L}
Let $L>0$. For $|\varepsilon|$ small enough,
\begin{enumerate} 
    \item $\mu_\varepsilon^L$ is an FFIID with exponential tails.
    \item There exists $c>0$ such that for all $l\in \N$, there exists an $l$--dependent distribution $\nu_l$ on $[-L,L]^{\Z^d}$ satisfying for every finite subset $\Lambda \subset \Z^d$,
        $$d_{\mathrm{TV}}(\mu^L_{\varepsilon}|_\Lambda, \nu_l |_\Lambda) \leq |\Lambda| e^{-cl}.$$ 
\end{enumerate}
\end{thm}

In Theorem \ref{thm_mu_L}, part 2 is a direct consequence of part 1 because of the exponential tails of the coding radius, according to \eqref{eq:distance_to_l-dep}.
Part 1 holds because when $|\varepsilon|$ is sufficiently small, $\pi^{\varepsilon, L}$ satisfies the high-noise condition \eqref{eq:high_noise} of Häggström and Steif.
They proved in \cite{HaggstromSteif} that high-noise Markov random fields are FFIID with exponential tails. Their result is stated for fields with finite state space, but the proof can be adapted to the compact case.
We present a different proof in this article, which can be seen as a preliminary version of the proof of Theorem \ref{thm_mu}.
Note that in \cite{ferrari}, the authors studied perfect sampling of truncated Gaussian vectors.
Our proof of Theorem \ref{thm_mu_L} is similar to their method, which they suggest in their final remark to implement for truncated Gaussian fields indexed by $\Z^d$, such as $\mu_\varepsilon^L$.

\paragraph{Idea of the proof.}
The proofs of Theorems \ref{thm_mu} and \ref{thm_mu_L} rely on a coupling-from-the past argument inspired by Propp and Wilson exact sampling algorithm \cite{propp_wilson}.
This is a classical technique to construct coding from i.i.d. processes, as described in \cite{VdB_steif}, \cite{HaggstromSteif}, \cite{spinka9}, \cite{spinka10}.
In essence, for a distribution $\mu$ on $\Omega$ (e.g. $\mu=\mu_\varepsilon$ or $\mu^L_\varepsilon$), we couple Markov processes of invariant law $\mu$ starting from all possible initial configurations.
We say that there is coalescence at a certain space-time position when all the processes coincide at this space-time position.
The evolution of the processes is governed by a dynamics analogous to the Glauber dynamics of the Ising model \cite{Glauber}.
Each transition consists of a new sampling of a single site of $\Z^d$, according to the conditional distribution given the rest of the configuration. We aim to maximise the \emph{coupling probability}, that is, the probability that a transition produces the same output at the updated site regardless of the configuration outside. This makes the processes coalesce, in the sense of Propp and Wilson. This allows us to get an exact sampling of distribution $\mu$, arising as an FIID.

To determine if there is coalescence at a certain space-time position, say $(0,0)$, we explore the dynamics backwards.
The value at $(0,0)$ in the different coupled processes is determined by the last update of site $0$. Say that the update is ``good'' if it makes all the processes coalesce. In this case, we can stop the exploration. Otherwise, the exploration continues from that time and all the neighbours of site $0$.
Therefore, the backward exploration is dominated by a Galton-Watson process of mean $|B| \times (1-$ maximal coupling probability).
This process extinguish almost surely if the maximal coupling probability is sufficiently large --- this happens exactly when $\mu$ satisfies the high noise assumption \eqref{eq:high_noise}. In this case, the value of the field at $(0,0)$ in the dynamics is independent of the initial configuration, if the initial time is chosen small enough. We will prove that $\mu_\varepsilon^L$ satisfies the high-noise assumption if $|\varepsilon|$ is small enough.

Then, we adapt the method to the case where the aforementioned maximal coupling probability is small, or even 0.
This is typically the case for $\mu_\varepsilon$.
The proposed approach involves stratification of the state space into different levels.
Rather than relying on the coalescence of potentially ``high'' values of the field, which is an unlikely outcome, we aim to decrease the level of the field first.
The idea is that during an update, if we know that in all the coupled Markov processes, the field on the neighbourhood of the updated site is below the level $n$, then it is likely that the new value of the field at the updated site falls below the level $n-1$; then, if all the values of the neighbours are below the lowest level, it is likely that the update makes the processes coalesce.
This translates into an interacting particle system that follows the knowledge about the level of the field we can get independently of the initial configuration.
This particle system admits a dual, which corresponds to the backward exploration of the dynamics, and that encodes the information required to achieve coalescence at the starting point. 
In contrast to the preceding case, where coalescence was likely within a single update, the backward exploration never stops: it is first necessary to ensure that the last update of site $0$ allows to coalesce from all possible values of the field below level 1; then, the second generation of updates seen in the exploration must allow the level to decrease from 2 to 1 etc. Therefore, it will not be true that we can have coalescence at $(0,0)$ with probability 1 by taking an initial time small enough, and we will need additional arguments.

\paragraph{Organisation of the paper.}
Theorems \ref{thm_mu} and \ref{thm_mu_L} will be in fact derived from more general results, respectively Theorems \ref{thm_general} and \ref{thm_compact}.
The Glauber dynamics and the coupling-from-the-past argument are described in Section \ref{Glauber dynamics for Gaussian model}. We also define the maximal coupling probability and compute it for $\mu_\varepsilon$ and $\mu_\varepsilon^L$.
The construction of a finitary coding under the high-noise condition is discussed in Section \ref{Proof of thm mu_L}, where we state and prove Theorem \ref{thm_compact}.
In Section \ref{Proof of thm mu}, we introduce the assumptions required for our stratification to work, define the aforementioned interacting particle system and its dual, and prove Theorem \ref{thm_general}.
Finally, in Section \ref{Proofs for the Gaussian models}, we prove Theorems \ref{thm_mu_L} and \ref{thm_mu} as consequences of Theorems \ref{thm_compact} and \ref{thm_general} respectively.

\paragraph{Acknowledgments.}
This work is part of a PhD thesis supervised by Vincent Beffara. I am grateful to him for introducing me to Glauber dynamics techniques, but above all for his ongoing support and guidance. This paper also owes a great deal to the anonymous referee for his feedback on the first version and his many suggestions for improvement.

\section{Glauber dynamics and maximal coupling probability}\label{Glauber dynamics for Gaussian model}
Throughout this Section, as well as Sections \ref{Proof of thm mu_L} and \ref{Proof of thm mu}, we fix $\mu$ the distribution of a stationary Markov random field $(X_i)_{i \in \Z^d}$, with boundary set $B \Subset \Z^d \setminus\{0\}$. Without loss of generality, we assume that $B$ is symmetric.

\subsection{Glauber dynamics}\label{Glauber dynamics}
In 1963, Roy J. Glauber \cite{Glauber} introduced a stochastic version of the Ising model, known as the \emph{stochastic Ising model} or \emph{Glauber dynamics} (see \cite{Holley} for an overview on the results for the Ising model).
It has now been generalised to a large class of statistical mechanics models.
Several versions of the Glauber dynamics coexist in the literature, the one presented and employed in this paper is also named the \emph{heat-bath} algorithm.
Roughly speaking, each site of $\Z^d$ is endowed with an ``exponential clock'', and when the clock at $i \in \Z^d$ rings, we update the value of the field at $i$ according to its conditional distribution given the rest of the configuration.
Below, we describe the \emph{graphical representation} (see \cite{Liggett1985InteractingParticleSystems} Chapter III.6 or \cite{durrett} Section 2) of the Glauber dynamics for $\mu$.

Let $\PP$ be a Poisson point process (PPP) of intensity $\# \otimes \d t \otimes \d u$, where $\#$, $\d t$, $\d u$ denotes the counting measure on $\Z^d$, the Lebesgue measure on $\R$ and the Lebesgue measure on $[0,1]$, respectively. Write $\P$ for the law of $\PP$. Points of $\PP$ are called the \emph{update marks} of the dynamics.
By properties of PPP, it means that sites of $\Z^d$ are independently endowed with a bi-infinite sequence $(T_n^\cdot,U^\cdot_n)_{n \in \Z}$, satisfying
$$\PP = \{(i,T^i_n, U^i_n), n \in \Z, i\in \Z^d\},$$
and such that $(T_n^\cdot)_{n \in \Z}$ is increasing.
This observation highlights that $\P$ is a product measure over $\Z^d$; we will build a factor map from this product measure to $\mu$.
For each $i\in \Z^d$, the $(T_n^i)_{n \in \Z}$ correspond to the times when site $i$ is updated in the dynamics, and the increments $(T_{n+1}^i - T_n^i)_{n \in \Z}$ are independent and have exponential distribution of parameter 1 (hence the terminology ``exponential clock'').
The $(U_n^i)_{i \in \Z^d, n \in \Z}$ are independent random variables distributed uniformly on $[0,1]$; the update of $i$ at time $T^i_n$ is made with the help of the extra randomness provided by $U_n^i$.
Suppose that we have an \emph{update function} $\varphi:S^B \times [0,1] \to S$, with the property that for all $\eta \in S^B$,
$$\int_0^1\ind{\varphi(\eta, u) \in \cdot } \d u= \mu(X_0 \in \cdot \mid (X_j)_{j \in B} = \eta),$$
meaning that if $U$ has uniform distribution on $[0,1]$ (denoted $\mathcal{U}[0,1])$, then $\varphi(\eta, U)$ has the same law as $X_0$ given that the field on $B$ coincide with $\eta$; note that by the Markov property, the conditional law of $X_0$ given $(X_i)_{i \neq0}$ depends only on the configuration on $B$. 
Then, for every update mark $(i,t,u)$, if the dynamics is in configuration $\xi \in \Omega$ before the update, then the new value of the field at $i$ is given by $\varphi( \xi |_{i+B}, u)$; the other values remain unchanged.
The property of the update function $\varphi$ guarantees that $\mu$ is preserved by the dynamics.

Note that we have defined the evolution of the dynamics on a time-line which is infinite in both directions. The graphical representation provides a coupling of processes starting from any time and any initial configuration, obtained by using the same PPP to construct them.
For $\tau \in \R$ and $\xi \in \Omega$, write $X^{\xi, \tau}:=(X^{\xi, \tau}_t)_{t \geq \tau}$ for the process starting at time $\tau$ with initial configuration $X^{\xi, \tau}_\tau = \xi$, and whose evolution is governed by the graphical representation of the Glauber dynamics.
More precisely, we have that for every $(i,t,u) \in \PP$ with $t>\tau$, a.s.
$$\forall j \neq i, X^{\xi, \tau}_t(j) = X^{\xi, \tau}_{t-}(j) \et X^{\xi, \tau}_t(i)=\varphi(X^{\xi, \tau}_{t-}|_{i+B},u).$$
Nevertheless, since there are infinitely many sites, we cannot define the first update after a given initial time $\tau \in \R$.
In other words, for every $\tau \in \R$, a.s.
$$\inf\{t > \tau : \exists i \in \Z^d, \exists u \in [0,1], (i,t,u) \in \PP \} = \tau.$$
An additional argument is therefore needed to ensure that the dynamics can indeed be constructed with the graphical representation.
The pitfall is that the first update of a given site (which is well defined) could require one to look infinitely far away.
Indeed, for each update one need to look at the neighbours, but it may happen that these neighbours needed to look at their own neighbours before etc.
A percolation argument of Harris \cite{HARRIS197266} (see also \cite{durrett} Section 2), handles this problem.
The idea is to take a random subgraph of $\Z^d$ defined as follows: fix a small positive number $t_0$;
for every $i,j \in \Z^d$ such that $i \sim j$, say that the edge $\{i,j\}$ is open if there exists $n\in \Z$ such that $\tau < T^i_n < \tau + t_0$ or $\tau < T^j_n < \tau + t_0$; otherwise, the edge is closed.
Hence, an open edge corresponds to the fact that an update of the system requires $i$ to look at $j$ or the reverse.
Consequently, if $i$ and $j$ are not in the same connected component in our random graph, then the values of the field at $i$ and $j$ at time $\tau + t_0$ are independent (assuming they were chosen independently at the initial time $\tau$).
Hence, one can compute separately the evolution up to time $\tau + t_0$ in each connected component. 
Harris's result is that for $t_0$ small enough, there is no infinite connected component for almost all realisation of the PPP.
Thus, each connected component have a first update time, so we can construct the dynamics up to time $\tau + t_0$.
Iterate the argument to construct the dynamics for all time $t\geq \tau$.
It follows that for any $t \geq \tau$, $X^{\xi, \tau}_t$ is a well-defined measurable function of the update marks between $\tau$ and $t$ and of the initial configuration.

\paragraph{Notations.}
In the following, we will often consider measurable events of the update marks. For two such events $A$ and $B$, we will write $A\as{\subset} B$ when for almost all realisation of the PPP, one has the inclusion $A \subset B$; in other words, $A\as{\subset} B$ means that $\P(A \setminus B)=0$.
Similarly, we write $A \as{=} B$ when both inclusion hold, that is when $\P(A \Delta B)=0.$

\subsection{Coupling-from-the-past}\label{Coupling-from-the-past}
The main idea to construct a coding for $\mu$ from the PPP is to rely on the \emph{coupling-from-the-past} algorithm, introduced by Propp and Wilson \cite{propp_wilson}, extended to the infinite-volume setting.
It allows to get an exact sampling of distribution $\mu$, by coupling Markov processes with invariant law $\mu$ started from all possible initial configuration at a random negative initial time: if the initial time is far enough from 0 so that all the processes take the same value at time 0, it follows that this common value has law $\mu$.

Considering the processes $(X^{\xi, \tau})_{\xi \in \Omega, \tau \in \R}$ defined above, we say that there is \emph{coalescence} from $\tau$ at space-time position $(i,t) \in \Z^d \times (\tau, +\infty)$ when $$\forall \xi, \xi' \in \Omega, X^{\xi, \tau}_t(i)=X^{\xi', \tau}_t(i).$$
This is a measurable event of the update marks between $\tau$ and $t$. Note that if we have coalescence at $(i,t)$ from $\tau$, it follows that we also have coalescence at $(i,t)$ from any earlier time $\tau' < \tau$.
For every $i \in \Z^d$, let
\begin{equation*}
    \widehat\tau_i:=\sup\{\tau <0 \mid \forall \xi, \xi' \in \Omega, X^{\xi, \tau}_0(i)=X^{\xi', \tau}_0(i)\}.
\end{equation*}
If $\widehat\tau_i$ is finite for all $i\in \Z^d$, one can define a random field $X^*=(X^*_i)_{i \in \Z^d}$ by
$$X_i^*=X_0^{\cdot, \widehat\tau_i}(i), \ i \in \Z^d,$$
noticing that one can take any initial configuration for the dynamics, since the value of $X_i^*$ is the same for any choice.
Note that $X^*$ is a measurable function of $\PP$, so it is an FIID. Then, the fact that $X^*$ is distributed according to $\mu$ follows from a classical coupling-from-the-past argument.
\begin{prop}\label{prop:CFTP}
    Assuming that a.s. $\widehat\tau_i$ is finite for all $i\in \Z^d$, the random field $X^*$ defined above has distribution $\mu$. Moreover, it gives a realisation of $\mu$ as an FIID.
\end{prop}
\begin{proof}
    Fix $\Lambda \Subset \Z^d$ and let $\xi$ be distributed according to $\mu$, independent of $\PP$.
    Since $\mu$ is invariant for the dynamics, it follows that for all $\tau<0$, $X^{\xi, \tau}_0$ has law $\mu$.
    On the other hand, for all $\tau \leq \min_{i \in \Lambda} \widehat\tau_i$, one has a.s. $X^{\xi, \tau}_0(i)=X_i^*$, by definition of the $\widehat\tau_i$. Note that $\min_{i \in \Lambda} \widehat\tau_i > -\infty$ a.s. because $\Lambda$ is finite. Therefore, $X^*|_\Lambda$ is distributed according to $\mu|_\Lambda$, and since $\Lambda$ was arbitrary, the result follows.
\end{proof}

\subsection{Maximal coupling probability}\label{Maximal coupling probability}
If we prove that the $\widehat\tau_i$ are finite almost surely, according to Proposition \ref{prop:CFTP}, it yields a coding of $\mu$ from the update marks.
To achieve that, we aim to have an update function $\varphi$ which favours coalescence.
It means that we want to maximise $$\mathbb{P}\big(\forall\eta, \eta' \in S^B, \varphi(\eta,U)=\varphi(\eta', U)\big)$$ for $U \sim \mathcal{U}[0,1]$, keeping the property that $\varphi(\eta,U)$ has law $\mu(X_0 \in \cdot \mid (X_j)_{j \in B}=\eta)$.

It is a general fact (see \cite{thorisson} Chapter 3, Theorem 7.1) that for any collection of probability measures $\pi^\eta$, defined on the same measurable space, for $\eta$ living in an arbitrary index set (e.g. $S^B$), we can define its greatest common component $\bigwedge_\eta \pi^\eta$ by
$$\text{ For any event }A, \Big(\bigwedge_\eta \pi^\eta \Big)(A)= \sup\{ \nu(A), \nu \leq \pi^\eta \ \forall \eta \}.$$
In general, $\bigwedge_\eta \pi^\eta$ is not a probability measure.
In fact, its total mass is the \emph{maximal coupling probability} (\cite{thorisson} Chapter 3, Theorem 7.2) of laws $\pi^\eta$, that is the maximal probability for realisations of $\pi^\eta$ to coincide in a coupling. Theorem 7.3 in \cite{thorisson} (Chapter 3) guarantees that there always exists a coupling that achieves this maximal coupling probability.

Then, returning to our setting, let
$$\gamma(\mu):= \max_\mathbb{P} \mathbb{P}(\forall \eta, \eta' \in S^B, X_i^\eta = X_i^{\eta'}),$$
where the maximum is taken over all the couplings of random variables $(X_i^\eta)_{\eta \in S^B}$ such that $X_i^\eta$ has law $\mu(X_i \in \cdot \, \mid (X_j)_{j \in i+B})(\eta)$. By the Markov property, this distribution is the same as $\mu(X_i \in \cdot \mid (X_j)_{j \neq i})(\xi)$ for any $\xi \in \Omega$ that coincides with $\eta$ on $i+B$.
Note that by stationarity, $\gamma(\mu)$ does not depend on $i$.
It follows that $\gamma(\mu)$ corresponds to the maximal coupling probability of laws $\pi^\eta:=\mu(X_0 \in \cdot \mid (X_j)_{j \in B}=\eta)$, for $\eta \in S^B$.
Then, it is clear that if $S$ is finite, $\gamma(\mu)$ corresponds to the multigamma admissibility of \cite{HaggstromSteif} (i.e. the left-hand side of \eqref{eq:high_noise}).

The best we can hope for the update function is that it returns the same value regardless of the first coordinate (which corresponds to the boundary condition) when the second coordinate belongs to a subset of $[0,1]$ of length $\gamma(\mu)$.
An explicit construction of an admissible update function is done in \cite{HaggstromSteif} for finite $S$, and see \cite{ferrari} or \cite{MurdochGreen} for the case of real-valued random variables, or the Appendix below.

We conclude this Section by computing the maximal coupling probability for our two Gaussian models.
For every $x \in \R$, set $g_x(t) := \frac{1}{\sqrt{2\pi}} e^{-(t-x)^2/2}$, the density of $\Norm(x,1)$.
The law $\Norm(x,1)$ conditioned to $[-L,L]$ (for $L>0$) is called the \emph{truncated to $[-L,L]$ normal distribution} of mean $x$ and variance 1; its density is given by
$$\overline{g_x}^L(t) := \dfrac{g_x(t)}{\int_{-L}^L g_x(s) \d s} \ind{t \in [-L,L]}.$$
We denote this distribution by $\overline{\Norm}^L(x, 1)$.

Let us start by computing $\gamma(\mu^L_\varepsilon)$, for $L>0$ and $|\varepsilon|<1/(2L^2)$. It is the maximal coupling probability of laws $\mu^L_\varepsilon(X_0 \in \cdot \mid (X_j)_{j \in B} =\eta)$ for $\eta \in [-L,L]^B$ (where $B=B^0_1\setminus\{0\}$). By definition of the model, $\mu^L_\varepsilon(X_0 \in \cdot \mid (X_j)_{j \in B} =\eta)$ coincides with $\overline{\Norm}^L(\varepsilon \sum_{j \in B} \eta_j, 1)$, so that $\gamma(\mu^L_\varepsilon)$ is the same as the maximal coupling probability of law $\overline{\Norm}^L( \varepsilon x, 1)$, for $x \in [-L,L]$.
According to \cite{thorisson} (Chapter 3, Corollary 8.1), the maximal coupling probability of a sequence of distributions admitting a density is given by the integral of the infimum of the densities.
Although $[-L,L]$ is uncountable, the result still holds for the collection $\overline{\Norm}^L( \varepsilon x, 1)$, $x \in [-L,L]$, and we have:
$$\gamma(\mu_\varepsilon^L) =  \int_{-L}^L \inf_{|x| \leq L } \overline{g_{\varepsilon x}}^L(t) \d t.$$
The inequality ``$\leq$'' follows from the aforementioned result, applied to the countable collection of law $\overline{\Norm}^L( \varepsilon x, 1)$, for $x \in \Q \cap[-L,L]$; the fact that it is an equality can be seen by density of $\Q$ and continuity of $\overline{g_{\varepsilon x}}^L$ with respect to $x$, or because we have an explicit construction of a coupling which matches this coupling probability (see \cite{ferrari}, \cite{MurdochGreen} or the Appendix).
It can be observed (see Lemma \ref{lemma_eta_L} below) that $\gamma(\mu_\varepsilon^L)$ goes to 1 when $\varepsilon$ goes to 0 ($L$ being fixed), so that $\mu_\varepsilon^L$ satisfies the high noise assumption \eqref{eq:high_noise} when $|\varepsilon|$ is small enough.

We now compute $\gamma(\mu_\varepsilon)$ for $|\varepsilon|\leq1$. It corresponds to the maximal coupling probability of $\Norm(\varepsilon x,1)$, for $x \in \R$. Hence, similarly as above, we have
$$\gamma(\mu_\varepsilon) =  \int_{-\infty}^{+ \infty} \inf_{x \in \R } g_{\varepsilon x}(t) \d t.$$
It can easily be observed that this quantity equals 0 for all $\varepsilon \neq0$.
This means that there is no chance to reach coalescence within a single update of the dynamics.

\section{Finitary coding with the high noise assumption}\label{Proof of thm mu_L}
This section is devoted to the proof of the following result, from which we will deduce Theorem \ref{thm_mu_L}.

\begin{thm}\label{thm_compact}
    If $\gamma(\mu) > 1-\frac{1}{|B|}$, then $\mu$ is an FFIID with exponential tails.
\end{thm}

We consider the dynamics introduced in Section \ref{Glauber dynamics for Gaussian model} that admits $\mu$ as invariant law, and the coupling of processes $(X^{\xi,\tau})_{\xi \in \Omega, \tau \in \R}$ defined with the graphical representation.
Recall that $\gamma=\gamma(\mu)$ is the maximal coupling probability of $\mu(X_0 \in \cdot \mid (X_j)_{j\in B}=\eta)$ for $\eta \in S^B$, so we suppose that the update function $\varphi$ verifies
\begin{equation}\label{maximal_coupling_phi_L}
    \forall \eta,\eta' \in S^B, \forall u \leq \gamma, \varphi(\eta,u) = \varphi(\eta',u)
\end{equation}
It means that for each update mark $(i,t,u) \in \PP$ with $u \leq \gamma$, the new value of the field at $i$ after the update is determined by $u$, independently of the values of the neighbours.
Hence, one has $X^{\xi, \tau}_t(i)=X^{\xi', \tau}_t(i)$ for any $\tau <t$ and $\xi,\xi' \in \Omega$, so the value of the field at $i$ at time $t$ becomes independent of the initial configuration: we have coalescence from $\tau$ at $(i,t)$.

\subsection{Spin system associated with the dynamics}\label{Spin system associated with the dynamics}
We associate with the dynamics a \emph{spin system}, that is, following the terminology of \cite{Liggett1985InteractingParticleSystems} (Chapter III), a continuous-time Markov process on $\{0,1\}^{\Z^d}$, for which only one coordinate changes in each transition.
It is constructed with the graphical representation, using the same update marks $\PP$ as before.
For any time $\tau \in \R$, let $(\omega^\tau_t)_{t\geq \tau}$ be the spin system started at $\tau$ with initial configuration $\omega^\tau_\tau \equiv 1$ (i.e. put spin 1 everywhere), and whose evolution is governed as follows:
for all $(i,t,u) \in \PP$, with $t> \tau$, the spin at $i$ is set to $1$ if $u>\gamma$ and $\omega^\tau_{t-}(j)=1$ for some $j \sim i$, and it is set to 0 otherwise; the other spins remain unchanged, i.e. we have
\begin{equation}\label{eq:transition_omega}
    \forall j\neq i, \omega^\tau_t(j)=\omega^\tau_{t-}(j) \et \omega_t^\tau(i)= \ind{u>\gamma}\ind{\exists j\sim i, \omega^\tau_t(j)=1}.
\end{equation}
Note that a transition occurs at time $t$ only if there is an update mark at that time, but not all the update marks generate an actual transition. Indeed, the value of a spin after an update can be the same as before.

For future purpose, we now give the rate function and the generator of the system.
For $\omega \in \{0,1\}^{\Z^d}$ and $i \in \Z^d$, call $\omega^i$ the configuration obtained from $\omega$ by flipping coordinate $i$, i.e. $$\omega^i(j) = \cas{\omega(j)}{\text{for } j \neq i}{1-\omega(j)}{\text{for }j=i}.$$
The rate function is defined for all $i \in \Z^d$ and $\omega \in \{0,1\}^{\Z^d}$ by
\begin{equation}\label{rate_fct}
    c(i, \omega) = \cas{\gamma + (1-\gamma) \ind{\forall j \sim i, \omega(j) = 0}}{\text{if } \omega(i) = 1}{(1-\gamma)\ind{\exists j\sim i, \omega(j) = 1}}{\text{if } \omega(i) = 0};
\end{equation}
$c(i,\omega)$ is the rate at which $\omega$ jumps to $\omega^i$. Roughly, it means that during an infinitesimal interval of time of length $\d t$, independently for every site $i$ of $\Z^d$, the spin at $i$ is flipped with probability $c(i,\omega) \d t$.
Then, the generator $\L$ of the system is defined by
\begin{equation}\label{generator_L}
    \L f(\omega) = \sum_{i \in \Z^d} c(i,\omega) (f(\omega^i) - f(\omega))
\end{equation}
for every \emph{local} function $f:\{0,1\}^{\Z^d} \to \R$, that is $f$ depends on finitely many coordinates.

The fact that $\L$ is indeed the generator of the systems $(\omega^\tau_t)_{t \geq \tau}$ constructed above with the graphical representation can be seen as follows.
The spin at site $i$ can be flipped at time $t$ only if there exists some update mark $(i,t,u) \in \PP$.
A flip from 1 to 0 is produced by the update mark if $\ind{u>\gamma}\ind{\exists j\sim i, \omega^\tau_t(j)=1}=0$, that is if $u \leq \gamma$ or if $u>\gamma$ and all the neighbours of $i$ have spin 0 at the time of the update.
Conversely, a flip from 0 to 1 is produced if $\ind{u>\gamma}\ind{\exists j\sim i, \omega^\tau_t(j)=1}=1$, that is if $u> \gamma$ and at least one neighbour of $i$ has spin 0 at the time of the update.
This is consistent with the rate function defined in \eqref{rate_fct}.

Observe that $c(\cdot,\cdot)$ is the rate function of an \emph{attractive} spin system, in the sense that
when $\omega_1 \leq \omega_2$\footnote{For the canonical partial order on $\{0,1\}^{\Z^d}$, i.e. $\omega_1 \leq \omega_2$ if and only if $\omega_1(j) \leq \omega_2(j)$ for every $j\in \Z^d$.}, for every $i\in \Z^d$,
$$c(i,\omega_1) \leq c(i, \omega_2) \text{ if } \omega_1(i)=\omega_2(i) =0,$$
$$c(i,\omega_1) \geq c(i, \omega_2) \text{ if } \omega_1(i)=\omega_2(i) =1.$$
This implies that for any $\tau \geq \tau'$, one has
$$\P(\forall t\geq \tau, \omega^{\tau'}_t \leq \omega^\tau_t)=1.$$
This is because our systems are initialised with the maximal initial condition, so that $\omega^{\tau'}_\tau \leq \omega^\tau_\tau$ almost surely, and this inequality propagates by attractiveness.

The interpretation of this spin system is that it encapsulates whether or not the value of the field in the dynamics can be computed in a measurable way only from the update marks, and not from the initial configuration; a spin 0 in a certain space-time position will mean that one almost surely has coalescence from $\tau$ at this space-time position.
More precisely, 

\begin{lemme}\label{lemme:spin_0_implies_coalescence}
For every $t\geq \tau$ and $i \in \Z^d$, we have
\begin{equation*}
    \{\omega^{\tau}_t(i) =0\} \as{\subset} \big\{\forall \xi, \xi' \in \Omega, X_t^{\xi, \tau} (i) = X_t^{\xi', \tau} (i)\big\}.
\end{equation*}
\end{lemme}
\begin{proof}
This is clearly true for all $i\in \Z^d$ at $t=\tau$ since the event on the left-hand side is empty, and one can check that the inclusion is preserved by the updates for almost all realisation of the PPP.
Suppose that $(i,t,u) \in \PP$ for some $t> \tau$ and that the inclusion holds for all $s\in[\tau,t)$ and all $j\in \Z^d$. By property of the PPP, a.s. there is no other update mark at time $t$.
We have $\omega^\tau_t(i)=0$ if and only if $u \leq \gamma$ or for all $j \sim i$, $\omega^\tau_t(j)=0$.
If the update mark satisfies $u \leq \gamma$, it follows from \eqref{maximal_coupling_phi_L} that we have coalescence at $(i,t)$. In the second case, the fact that $\omega^\tau_t(j)=0$ for all $j\sim i$ implies, by assumption, that the values of the field on the neighbourhood of $i$ do not depend on the initial configuration of the dynamics; since the output of the update function $\varphi$ only depends on the values of the field on this neighbourhood, it means that this output itself does not depend on the initial configuration; therefore, we get that $X_t^{\xi, \tau} (i) = X_t^{\xi', \tau} (i)$ for all $\xi, \xi' \in \Omega.$
\end{proof}
For every $i \in \Z^d$, let
\begin{equation*}
    \tau_i=\sup\{\tau <0 \mid \omega^\tau_0(i)=0\}.
\end{equation*}
Observe that $\tau_i$ is a random time, measurable with respect to $\PP$.
By attractiveness, it follows that if $\tau \leq \tau_i$, one has a.s. $\omega_0^\tau(i)=0$.
Moreover, Lemma \ref{lemme:spin_0_implies_coalescence} implies that almost surely $\tau_i \leq \widehat\tau_i$ for all $i \in \Z^d$. Therefore, it is sufficient to show that $\tau_i>-\infty$ almost surely to be able to define $X^*$ as in Section \ref{Coupling-from-the-past}, and obtain a coding for $\mu$, according to Proposition \ref{prop:CFTP}.

Our main goal in the rest of this Section is to prove almost sure finiteness of the $\tau_i$.
Since they are identically distributed, it is enough to prove it for $\tau_0$ (0 being the origin of $\Z^d$).
The proof strategy is to dominate $\tau_0$ by the extinction time of a branching process. For that purpose, we formally define the backward exploration mentioned in the Introduction.

\subsection{Backward exploration of the dynamics}\label{Backward exploration of the dynamics}
We introduce the notion of active paths, borrowed from \cite{Liggett1985InteractingParticleSystems}, which represents the paths that one follows while exploring the dynamics backward. In the following, for an update mark $\rho \in \PP$, write $U_\rho$ for the third coordinate of $\rho$. Abusing the notation, we will sometimes not distinguish an update mark from its space-time position (i.e. its first two coordinates), as almost surely there is at most one mark at any time $t$ (in this case, we will often drop the word ``update'' and call a mark the space-time position of an update mark).
For every $(i,t) \in \Z^d \times \R$, if there is an update mark at $(i,t)$, put an oriented \emph{arrow} from $i$ to $j$ at time $t$ for all $j \in i+B$.
For $i,j \in \Z^d$, $t > s$, say that there is an \emph{active path} of length $n$ from $(i,t)$ to $(j,s)$ if there is a sequence of times $t=t_0 > t_1 > \dots > t_n \geq t_{n+1}=s$ and a sequence of sites $i=x_0, x_1, \dots, x_n=j$ such that:
\begin{enumerate}[label = (\roman*)]
    \item for every $k\in \{1, \dots n\}$, there is an arrow between $x_{k-1}$ and $x_k$ at time $t_k$
    \item for every $k\in \{0, \dots n\}$, there is no mark in the vertical segment $\{x_i\} \times (t_i, t_{i+1})$.
\end{enumerate}
In other words, active paths go backward in time, and each time they find a mark, they necessarily choose an arrow in the direction of a neighbour of their current location.
The length corresponds to the number of crossed arrows, equivalently the number of marks discovered.
Note that the active paths do not depend on the third coordinate of the update marks, but only on their space-time position.

For every $(i,t) \in \Z^d \times \R$, let $\rho(i,t)$ be the last mark at site $i$ before time $t$:
\begin{equation}\label{rho}
    \rho(i,t) = (i, \sup\{s \leq t : \exists u \in [0,1],  (i,s,u) \in \PP \}),
\end{equation}
the supremum being a.s. finite by property of the PPP.
From the definition of an active path, we derive a notion of distance\footnote{Note that it is not symmetric; it should be understood as a directed graph distance.} on $\PP$.
For two marks $\rho$ and $\rho'$ in $\PP$, there may be several active paths from $\rho$ to $\rho'$. Define $d(\rho, \rho')$ as the minimal length of such a path.
If there is none, set $d(\rho, \rho') := + \infty$.

For each $\rho\in \PP$, we define the \emph{active marks} from $\rho$ to be the ones to which there is an active path from $\rho$. The set $\A_\rho$ of active marks from $\rho$ is therefore given by
\begin{equation}\label{A_rho}
    \A_\rho = \{\rho' \in \PP ~\vert~ d(\rho, \rho') < +\infty\}.
\end{equation}

In a realisation of the dynamics, the value of the field at $(i,t)$ is therefore a measurable function of the update marks of $\A_{\rho(i,t)}$ and of the initial configuration.
Coalescence is achieved when this dependence on the initial configuration disappears.

We consider an oriented Bernoulli site percolation on the marks.
For every $\rho \in \PP$, say that $\rho$ is \emph{open} if $U_\rho > \gamma$.
An \emph{open active path} is an active path that crosses only open marks.
It is equivalent to consider that we keep the arrows originating from $\rho$ only if $U_\rho>\gamma$.
The orientation of the active paths gives an orientation to the percolation model.

For all $t \leq 0$, let $\hat \sigma_t$ be the set of sites $j \in \Z^d$ such that there exists an open active path from $(0,0)$ to $(j,t)$, that is
$$\hat \sigma_t = \{j \in \Z^d: (0,0) \xrightarrow[]{o.a.} (j,t) \},$$
where o.a. stands for open active.
Note that $(\hat \sigma_t)_{t\leq 0}$ is a measurable function of $\PP$ and that, by definition of the active paths, it is a right-continuous Markov process.

The following lemma makes the connection with the spin systems defined in the previous paragraph.
In the statement and in the proof, each element of $\{0,1\}^{\Z^d}$ is identified with the subset of $\Z^d$ where it takes value $1$.

\begin{lemme}\label{couplage_disjoint}
    For all $\tau<0$, one has $$\{\omega^\tau_0(0)=0\} \as{\subset} \{\forall t \in [\tau, 0], \omega^\tau_t \cap \hat \sigma_t = \emptyset\}.$$
    In particular, $\{\omega^\tau_0(0)=0\} \as{\subset} \{\hat \sigma_\tau = \emptyset\}.$
\end{lemme}
\begin{proof}
    For $t=0$, we have a.s. $\hat \sigma_0 = \{0\}$, so $\{0\notin \omega_0^\tau\} \as{\subset}\{\omega^\tau_0 \cap \hat \sigma_0 = \emptyset\}$.
    Thus, we only need to check that the property $\omega^\tau_t \cap \hat \sigma_t = \emptyset$ is preserved a.s. by every update of the systems.
    Actually, since we already have the inclusion for the terminal time $0$ (and not the initial time $\tau$), we need to look at the updates backward.

    Suppose that $(i,t,u) \in \PP$ for some $t \in [\tau, 0]$ and that $\omega^\tau_t \cap \hat \sigma_t=\emptyset$.
    Show that $\omega^\tau_{t-} \cap \hat \sigma_{t-}=\emptyset$.
    First, since we have a mark at $(i,t)$, we always have $\hat \sigma_{t-}(i) =0$ (any potential active path from $(0,0)$ to $(i,t)$ will jump to a neighbour of $i$ at time $t$). Hence $i \notin \omega^\tau_{t-} \cap \hat \sigma_{t-}$.
    Then, transitions on the system $(\omega^\tau_t)_{t\geq \tau}$ can only flip the spin of the site that carries the update mark, so we have $\omega^\tau_{t-} \subset \omega^\tau_t\cup \{i\}$.
    The update mark $(i,t,u)$ extends an open active path if and only if 
    $i \in \hat \sigma_t$ and the mark is open (i.e. $u>\gamma$).
    If one of these two conditions fails, we have $\hat \sigma_{t-} = \hat \sigma_t - \{i\}$, so $\omega^\tau_{t-} \cap \hat \sigma_{t-} = \emptyset$.

    When the mark is open and $i \in \hat \sigma_t$, it yields an open active paths from $(0,0)$ to $(j,t-)$ for every neighbour $j$ of $i$.
    Hence, every $j \in i+B$ belongs to $\hat \sigma_{t-}$. 
    Therefore, we only need to check that we cannot simultaneously have $\omega^\tau_{t-}(j)=1$.
    By contradiction, if some $j$ verifies $\omega^\tau_{t-}(j)=1$, as one has $u>\gamma$, the transition rule \eqref{eq:transition_omega} implies that $\omega^\tau_t(i) =1$.
    This contradicts the fact that $\omega^\tau_t \cap \hat \sigma_t=\emptyset$, because $i \in \hat \sigma_t.$

    At time $t=\tau$, the property tells us that $\hat \sigma_\tau = \emptyset$, because $\omega^\tau_\tau = \Z^d$.
\end{proof}

\subsection{Duality between the spin system and the backward exploration}\label{Dual process}

Now, our next step is to prove the reverse inclusion $\{\omega^\tau_t(0)=0\}\as\supset\{\hat \sigma_\tau=\emptyset\},$ to obtain that $\tau_0$ coincides a.s. with the extinction time of the process $\hat \sigma_t$, that is the first (negative) time when it is empty. We will then be able to prove that this extinction time is almost surely finite, by comparison with a branching process.
The reverse inclusion is obtained by \emph{duality}, which is a powerful tool in the study of spins systems.
The following general definition, that can be found in \cite{Liggett1985InteractingParticleSystems}, holds for a larger class of interacting particle systems.
\begin{defi}\label{def:duality}
Suppose that $(\omega_t)_{t\geq \tau}$ and $(\sigma_t)_{t\geq \tau}$ are Markov processes on $\Omega$ and $\Sigma$ respectively, started at time $\tau \in \R$. Let $H$ be a bounded measurable function on $\Omega \times \Sigma$. 
The processes $(\omega_t)_{t\geq \tau}$ and $(\sigma_t)_{t\geq \tau}$ are said to be \emph{dual} with respect to $H$ if for all time $t\geq \tau$,
\begin{equation}\label{def_duality}
    \E_1^\omega H(\omega_t, \sigma)  = \E_2^\sigma H(\omega, \sigma_t) 
\end{equation}
for all $\omega \in \Omega$ and $\sigma \in \Sigma$, where $\E_1^\omega$ (resp. $\E_2^\sigma$) denotes the expectation with respect to the law of the process $(\omega_t)_{t\geq \tau}$ (resp. $(\sigma_t)_{t \geq \tau}$) started at time $\tau$ from configuration $\omega$ (resp. $\sigma$).
\end{defi} 

We will take as duality function the function $H(\omega, \sigma) = \prod_{i \in \Z^d} (1-\omega(i) \sigma(i))$.
It is frequently used in the study of spins systems, and is known as the coalescing duality.
By identifying $\omega$ and $\sigma$ with the set of sites where they take value one, one can rewrite $H$ as $H(\omega, \sigma) = \ind{\omega \cap \sigma = \emptyset}$.

We aim to define the dual of the spin system introduced in Section \ref{Spin system associated with the dynamics}. It will take values in the subset of $\{0,1\}^{\Z^d}$ of configurations containing finitely many 1, denoted $\Sigma$.
The Hille-Yosida theory on semigroups of operators and their generators enables us to prove duality merely by computation, based on the definitions of the spin systems by their generators.
We therefore first abstractly define the generator of a spin system on $\Sigma$, and provide some insight into why it is a good candidate for being the generator of the duals systems of $(\omega^\tau_t)_{t \geq \tau}$.
We then formally prove the duality relying on Theorem 3.42 in \cite{liggett2010continuous} (which is a consequence of the Hille-Yosida theorem), and show that $(\hat \sigma_{\tau-t})_{t \geq \tau}$ is a version of the dual of $(\omega^\tau_t)_{t\geq \tau}$.

For $\sigma \in \Sigma$ and $i \in \Z^d$, $\sigma^i$ still denotes the configuration obtained from $\sigma$ by flipping the spin at $i$, and $\sigma^{i,B}$ is the configuration obtained from $\sigma$ by replacing the spin of each neighbour of $i$ by 1 and the spin of $i$ by 0, i.e.
$$\sigma^{i,B} (j) = \begin{cases}
1& \si j \sim i\\
0&  \si j=i\\
\sigma(j) & \text{ otherwise}\end{cases}$$
Let $\M$ be the generator of a spin system on $\Sigma$ satisfying for every $f:\Sigma \to \R$
\begin{equation}\label{generator_M}
    \M f(\sigma) = \sum_{i \in \Z^d} \sigma(i) \left[(1-\gamma) f(\sigma^{i,B}) + \gamma f(\sigma^i) - f(\sigma)\right].
\end{equation}
The fact that elements of $\Sigma$ are 0 almost everywhere ensures that this expression is well-defined for every function $f$.

Intuitively, the duality can be understood with the graphical representation.
Consider a \emph{backward process} $(\sigma_t)_{t \leq 0}$ as, i.e. a process that goes back in time. Thus, it reads the update marks in the other way than the spin systems of Section \ref{Spin system associated with the dynamics}. It is initialised with configuration $\sigma_0=\{0\}$ (that is the configuration with spin 0 everywhere except at the origin of $\Z^d$).
From a configuration $\sigma$, within a single transition, it can only reach configurations of the form $\sigma^i$ or $\sigma^{i,B}$.
A transition can occur at time $t$ only if there exists an update mark $(i,t,u) \in \PP$.
Note that since we consider $(\sigma_t)_{t \leq 0}$ as a backward process, it will be left-continuous, contrary to the usual.
If $(i,t,u)\in \PP$ and $\sigma(i)$ denotes the spin at $i$ before the update, that is informally at time $t+$, the transition rules are the following:
\begin{itemize}
    \item If $\sigma(i) =0$, nothing happens.
    \item If $\sigma(i) =1$ and $u\leq \gamma$, jump to $\sigma^i.$
    \item If $\sigma(i) =1$ and $u> \gamma$, jump to $\sigma^{i,B}.$
\end{itemize}
It should be clear that $\M$ is the generator of $(\sigma_t)_{t \leq 0}$.
One may think about a contact process:
the transition is allowed only if $\sigma(i) = 1$ (equivalent of an infected site) and a jump to $\sigma^i$ corresponds to the recovery of the particle in $i$, whereas a jump to $\sigma^{i,B}$ corresponds to the infection of the neighbours (and simultaneously the recovery of the particle).

The interpretation of this process is that $\sigma_t$ represents what we need at time $t$ to obtain almost sure coalescence at $(0,0)$ from a certain negative time:
$\sigma_t(i) = 1$ means that we already need to have coalescence at $(i,t)$ almost surely.
This is why the process is initialised at time 0 with spin 0 everywhere except at the origin.
Given this interpretation, one can understand the transition rules as follows:
going back in time, suppose that we meet an update mark $(i,t,u)$ and that $\sigma_{t+}(i) = 1$.
If the update mark is bad, i.e. $u>\gamma$, the output of $\varphi$ in the dynamics depends on the values of the neighbours at time $t$; to get that this output is independent of the initial configuration, we therefore need that coalescence occurred for all these neighbours before $t$; this is why we allocate spin 1 in $\sigma_t$ to every neighbour of $i$.
If the update mark is good, i.e. $u>\gamma$, the values of the neighbours are unneeded to compute the output of $\varphi$, so we only set $\sigma_t(i) = 0$.
On the other hand, when the configuration before the update verifies $\sigma_{t+}(i) = 0$, it means that we do not need any information on the value of the field at $(i,t)$ in the dynamics, so we do not have to make a transition in the process $(\sigma_t)_{t\leq 0}$.

We can now forget about the specific process $(\sigma_t)_{t \leq 0}$, which has been introduced only to provide an intuitive understanding of duality. We now prove that $\M$ generates spin systems that are dual to the $(\omega^\tau_t)_{t \geq \tau}, \ \tau \in \R$.

\begin{lemme}\label{prop_duality}
The spin systems generated by $\L$ and $\M$ are dual with respect to $H(\omega, \sigma) = \ind{\omega \cap \sigma = \emptyset}.$
\end{lemme}

\begin{proof}
According to Theorem 3.42 in \cite{liggett2010continuous}, we only need to check that $$\L H(\cdot, \sigma)(\omega) = \M H(\omega, \cdot)(\sigma)$$ for every $\omega\in \{0,1\}^{\Z^d}$ and $\sigma \in \Sigma$, providing these two quantities have a meaning.

Let $\omega \in \{0,1\}^{\Z^d}$ and $\sigma \in \Sigma$.
The fact that we can apply $\L$ to $H(\cdot, \sigma)$ follows from the fact that $\sigma$ have finitely many 1, so $H(\cdot, \sigma)$ is a local function; $H(\omega, \cdot)$ can be applied to $\M$, whose action is defined on every function.
It is therefore enough to show that for every site $i \in \Z^d$, 
\begin{equation}\label{duality}
    c(i,\omega) \left[H(\omega^i, \sigma) -H(\omega, \sigma) \right] = \sigma(i) \left[ (1-\gamma) H(\omega, \sigma^{i,B}) + \gamma H(\omega, \sigma^i) - H(\omega, \sigma) \right],
\end{equation}
where $c(i,\omega)$ is given by \eqref{rate_fct},
since we can then sum over $i \in \Z^d$ to get $\L H(\cdot, \sigma) (\omega) = \M H(\omega,\cdot) (\sigma)$, which is what we want to prove.

For every $i \in \Z^d$, we have
\begin{enumerate}[label=(\roman*)]
    \item $H(\omega^i, \sigma) - H(\omega, \sigma) = \cas{1}{\text{if } \omega(i) = \sigma(i) = 1 \text{ and } \forall j \neq i, \omega(j) \sigma(j) = 0}{-1}{\text{if } \omega(i) = 0, \sigma(i) = 1 \text{ and } \forall j, \omega(j) \sigma(j) = 0};$
    
    \item $H(\omega, \sigma^i) - H(\omega, \sigma) = \cas{1}{\text{if } \omega(i) = \sigma(i) = 1 \text{ and } \forall j \neq i, \omega(j) \sigma(j) = 0 }{-1}{\text{if }\omega(i) = 1, \sigma(i) = 0 \text{ and } \forall j , \omega(j) \sigma(j) = 0};$
    
    \item $H(\omega, \sigma^{i,B}) - H(\omega, \sigma) = \cas{1}{\text{if } \omega(i) = \sigma(i) = 1, \forall j \sim i, \omega(j) =0 \text{ and } \forall j \neq i, \omega(j) \sigma(j) = 0}
    {-1}{\text{if } \exists j \sim i, \omega(j) = 1 \text{ and } \forall j, \omega(j) \sigma(j) = 0}.$
\end{enumerate}
In those three equations, in all other cases, the left-hand side equals zero.
Thus, in particular, if there exists $j \neq i$ such that $\omega(j) \sigma(j) = 1$, both sides of \eqref{duality} are null. It is the same if $\sigma(i)=0$, since (i) equals 0 in this case.
We therefore assume that $\sigma(i) = 1$ and that for every $j \neq i$, $\omega(j) \sigma(j) = 0$, and check the equality in \eqref{duality}.
\begin{itemize}
    \item \underline{If $\omega(i) = 1$ :}
In this case, with our assumptions, (i) and (ii) equals 1, whereas (iii) equals $\ind{\forall j \sim i, \omega(j) = 0}$.
Thus, the left-hand side of \eqref{duality} equals $\gamma+(1-\gamma) \ind{\forall j \sim i, \omega(j) = 0}$
and the right-hand side $\sigma(i) \left[(1-\gamma) \ind{\forall j \sim i, \omega(j) = 0} +\gamma \right]$, so they coincide.

    \item \underline{If $\omega(i) = 0$ :}
Now, (i) is $-1$, (ii) is 0 and (iii) equals $-\ind{\exists j \sim i, \omega(j)=1}$, so both sides of \eqref{duality} equal $-\eta \ind{\exists j \sim i, \omega(j)=1}$.\qedhere
\end{itemize} \end{proof}

We now prove the key result of this paragraph. We do the same identification as before between subsets of $\Z^d$ and elements of $\{0,1\}^{\Z^d}$, this time in the other direction.

\begin{prop}\label{prop:duality_spin_system_backward_exploration}
    For all $\tau\in \R$, $(\omega^\tau_t)_{t\geq \tau}$ and $(\hat \sigma_{\tau-t})_{t \geq \tau}$ are dual with respect to $H(\omega, \sigma)=\ind{\omega \cap \sigma=\emptyset}.$
\end{prop}
\begin{proof}
    Fix $\tau \in \R$. According to Lemma \ref{prop_duality}, one only has to verify that the generator of $(\hat \sigma_{\tau-t})_{t \geq \tau}$ is $\M$. Recall that this process has been defined through the graphical representation, and note that it is left-continuous.
    To understand the transition mechanism of $(\hat \sigma_{\tau-t})_{t \geq \tau}$, we look at the transition encoded by some update mark $\rho=(i,\tau-t,u) \in \PP$, for some $t\geq \tau$.
    If $\hat \sigma_{\tau-t}(i)=0$, there is no open active path from $(0,0)$ to $(i, \tau-t)$ so the update mark does not affect the process.
    Otherwise, when $\hat \sigma_{\tau-t}(i)=1$, there are two cases, depending on whether $\rho$ is open or not:
    if $\rho$ is closed, the open active path stops at $(i,\tau-t)$ so $i$ does not belong to $\hat \sigma_{(\tau-t)-}$;
    that corresponds to the transition from $\sigma$ to $\sigma^i$ when $u \leq \gamma$;
    if $\rho$ is open, it yields an open active path from $(0,0)$ to $(j,(\tau - t)-)$ for all neighbours $j$ of $i$ (and no more active path from $(0,0)$ to $(i,(\tau-t)-)$); hence $\hat\sigma_{(\tau-t)-}(j) = 1$ for every $j\sim i$ (and $\hat\sigma_{(\tau-t)-}(i)=0$); that corresponds to the transition from $\sigma$ to $\sigma^{i,B}$ when $u>\gamma$.

    Therefore, for every $i \in \Z^d$, transitions in $(\hat \sigma_{\tau-t})_{t \geq \tau}$ from the current configuration $\sigma$ to $\sigma^i$ appear at rate $\gamma$, but are discarded if $\sigma(i)=0$; similarly, transitions from $\sigma$ to $\sigma^{i,B}$ appear at rate $1-\gamma$, but are discarded if $\sigma(i)=0$. These are exactly the rates prescribed by $\M$.
\end{proof}

The duality relation \eqref{def_duality} applied to $(\omega^\tau_t)_{t\geq \tau}$ and $(\hat \sigma_{\tau-t})_{t \geq \tau}$ at time $t=0$ therefore gives
$$\P(\omega_0^\tau(0) = 0) = \P(\hat \sigma_\tau = \emptyset),$$
recalling that $\omega^\tau_\tau \equiv1$ and $\hat \sigma_0=\{0\}.$
Together with the first inclusion, given by Lemma \ref{couplage_disjoint}, we get for all $\tau \in \R$
\begin{equation}\label{eq:duality_bilan}
    \{\omega_0^\tau(0) = 0\}\as{=}\{\hat \sigma_\tau = \emptyset\}.
\end{equation}
This characterises $\tau_0$ as the extinction time of the backward exploration process $(\hat \sigma_t)_{t \leq 0}$.
In the following paragraph, we prove that this process dies out in finite time almost surely, which yields a coding of $\mu$, by Proposition \ref{prop:CFTP}. We then explain how to control the tail of the coding radius to conclude the proof of Theorem \ref{thm_compact}.

\subsection{Proof of Theorem \ref{thm_compact}}
We start by proving that the backward exploration terminates in finite time. As mentioned in the Introduction, the proof can be summarised as a domination by a Galton-Watson process of reproduction law $\gamma \delta_0+(1-\gamma)\delta_{|B|}$. 
\begin{lemme}\label{extinction}
    If $\gamma > 1-\frac{1}{|B|}$, then almost surely there exists $\tau>-\infty$ such that $\hat \sigma_\tau = \emptyset$ (i.e. $(\hat \sigma_t)_{t \leq 0}$ dies out).
\end{lemme}

\begin{proof}
    The survival of $(\hat \sigma_t)_{t \leq 0}$ is equivalent to the existence of an infinite open active path starting from $(0,0)$.
    An infinite active path starting from $(0,0)$ is identified with the sequence $(\rho_n)_{n \in \N}$ of marks it crosses (hence one has $\rho_0 = \rho(0,0)$, defined in \eqref{rho}). 
    For any $n \in \N$, there exists an open active path of length $n$ starting from $(0,0)$ if and only if there exist $n$ open marks $(\rho_0, \dots, \rho_{n-1})$ with $\rho_0 = \rho(0,0)$ that correspond to the first $n$ marks of an infinite active path.
    Since two consecutive marks in the sequence must be carried by neighbours, there are at most $|B|^{n-1}$ such sequences.
    
    Each mark is open with probability $1-\gamma$.
    Moreover, the third coordinates of the update marks crossed by an active path are independent, by property of the PPP. 
    Therefore, by a union bound, we get
    \begin{equation}\label{active_path_n}
        \P\left(\exists \text{ an open active path of length $n$ starting at $(0,0)$}\right) \leq |B|^{n-1} (1-\gamma)^n.
    \end{equation}
    This quantity tends to $0$ as soon as $1-\gamma < \frac{1}{|B|}$.
    In this case, the probability of survival of $(\hat \sigma_t)_{t \leq 0}$ falls to 0, so almost surely $(\hat \sigma_t)_{t \leq 0}$ dies out in finite time.    
    \end{proof}

\begin{proof}[Proof of Theorem \ref{thm_compact}]
Combining the results of Lemma \ref{couplage_disjoint} and Proposition \ref{prop:duality_spin_system_backward_exploration}, we already noticed that 
$$\{\omega_0^\tau(0) = 0\}\as{=}\{\hat \sigma_\tau = \emptyset\},$$see \eqref{eq:duality_bilan}.
By Lemma \ref{extinction}, this implies that $\tau_0$ is almost surely finite. Then, Lemma \ref{lemme:spin_0_implies_coalescence} and Proposition \ref{prop:CFTP} yields a realisation $(X_i^*)_{i\in \Z^d}$ of $\mu$ as a measurable function of $\PP$, that is we have a coding for $\mu$.

We finish the proof of Theorem \ref{thm_compact} by studying the coding radius of this coding.
The coding radius $R$ is defined as the (random) minimal integer $r$ such that updates marks carried by sites in $B^0_r$ only are used to generate $X_0^*$.
The update marks used to generate $X^*_0$ are the open ones that are crossed by an open active path from $(0,0)$ and the closed ones at which such a path end up.
Indeed, the value of the field at the origin at time 0 in the dynamics started at time $\tau <0$ is a priori a measurable function of the update marks in $\A_{\rho(0,0)}$ and of the initial configuration.
However, if $\tau$ is chosen lower than $\tau_0$, then all the open active paths end up at a closed mark before reaching time $\tau$.
For these closed marks, the output of the update function $\varphi$ is independent of the initial configuration. Then, it follows that the values of the field in the dynamics along all these active paths are independent of the initial configuration.


Let $r=\max \{||j||_1, j \in B\}$. A step of length 1 in an active path cannot travel a distance greater that $r$ in $\Z^d$.
Thus, an update mark carried by a site outside $B_{rn}^0$ needs to be revealed to generate $X_0^*$ only if there is an open active path of length at least $n$ starting from $(0,0)$. Then, we derive from \eqref{active_path_n} an upper bound on the tail of $R$, since
$$\P(R \geq rn) \leq \P\left(\exists \text{ an open active path of length $n$ starting at $(0,0)$}\right) \leq |B|^{n-1} (1-\gamma)^n.$$
It shows the exponential decay of the tail of the coding radius when $1-\gamma < \frac{1}{|B|}$.
\end{proof}

\section{Coding without the high noise assumption}\label{Proof of thm mu}
As discussed in Section \ref{Maximal coupling probability}, for $\mu=\mu_\varepsilon$, we have $\gamma(\mu) = 0$, so the previous theorem cannot apply (except for $\varepsilon =0$ which actually corresponds to the i.i.d. Gaussian field).
We therefore need a new strategy to get a non-trivial coupling probability.

Let $(S_n)_{n \geq1}$ be an increasing sequence of subsets of $S$ such that $S_n \nearrow S$, $o$ a distinguished point in $S_1$, and $(q_n)_{n \geq0}$ an increasing to 1 sequence in $[0,1]$.
We call $\mathbf{0}$ the configuration of $\Omega$ with $o$ everywhere; abusing the notation, we also denote by $\mathbf{0}$ its restriction to $B$.
Let
$$\tilde \gamma(\mu) = \max_\mathbb{P} \mathbb{P}(\forall \eta \in S_1^B, X^\eta = X^{\mathbf 0}\text{ and } X^{\mathbf 0} \in S_1),$$
where the maximum is taken over all the couplings of random variables 
$(X^\eta)_{\eta \in S_1^B}$ such that $X^\eta$ has law $\pi^\eta=\mu(X_0 \in \cdot \mid X|_B=\eta)$.
Here, $\tilde \gamma(\mu)$ corresponds to the maximal probability for realisations of
$\pi^\eta$, for $\eta \in S_1^B$, to coincide in a coupling and belong to $S_1$. 

Like in the previous section, we prove a general result of which Theorem \ref{thm_mu} is a special case.

\begin{thm}\label{thm_general}
    Suppose that $(S_n)_{n \geq1}$ and $(q_n)_{n \geq 0}$ as above satisfy the four following hypotheses:
\begin{equation*}\tag{H1}\label{link_L0_p0}
    \tilde \gamma(\mu) \geq q_0,
\end{equation*}
\begin{equation*}\tag{H2}\label{link_Ln_pn}
    \forall n \geq 1, \; \forall \eta \in S^B_{n+1}, \; \mu(X_0 \in S_n \mid X|_B=\eta) \geq q_n,
\end{equation*}
\begin{equation*}\tag{H3}\label{condition_qn}
        8\sum_{n \geq 0} |B|^{2n+1} (1-q_n) <1 \text{ and }\exists t>0,  \sum_{n \geq 0} |B|^n e^{t|B|^n} (1-q_n) < \infty,
    \end{equation*}
\begin{equation*}\tag{H4}\label{condition_Ln}
    \mu(\{\xi \in \Omega: \forall i \in \Z^d, \xi_i \in S_{\max(n, ||i||)} \})\xrightarrow[n\to +\infty]{}1.
\end{equation*}
    Then, $\mu$ is an FIID and there exists $c>0$ such that for all $l\in \N$, there exists an $l$--dependent distribution $\nu_l$ on $\Omega$ satisfying for every finite subset $\Lambda \subset \Z^d$,
        $$d_{\mathrm{TV}}(\mu|_\Lambda, \nu_l |_\Lambda) \leq |\Lambda| e^{-cl}.$$
\end{thm}

Note that in the above statement, we say nothing about the tail of the coding radius. The rest of Section \ref{Proof of thm mu} is devoted to the proof of Theorem \ref{thm_general}.

\subsection{Interacting particle system associated with the dynamics}\label{Interacting particle system associated with the dynamics}
Similarly as for the proof of Theorem \ref{thm_compact}, we consider the dynamics introduced in Section \ref{Glauber dynamics for Gaussian model} that admits $\mu$ as invariant law, and the coupling of processes $(X^{\xi,\tau})_{\xi \in \Omega, \tau \in \R}$ defined with the graphical representation.
Assume that there exist $(S_n)_{n \geq 1}$, referred to as the \emph{levels}, and $(q_n)_{n \geq 0}$ such that \eqref{link_L0_p0}, \eqref{link_Ln_pn}, \eqref{condition_qn} and \eqref{condition_Ln} hold.

The first two hypotheses result in two properties on the update function $\varphi$:
\begin{equation}\label{maximal_coupling}
    \forall \eta \in S_1^B, \forall u \in [0,q_0], \varphi(\eta,u) = \varphi(\mathbf{0},u) \text{ and } \varphi(\mathbf{0},u) \in S_1
\end{equation}
and 
\begin{equation}\label{coupling_q_L}
    \forall n \geq 1, \forall \eta \in S^B_{n+1}, \forall u \in [0,q_n], \varphi(\eta,u) \in S_n.
\end{equation}
We explain in the Appendix how to construct an update function that satisfies these properties.

This means that in a realisation of the dynamics, for each update mark $(i,t,u) \in \PP$, if we happen to know that the values of the field on the neighbourhood of $i$ belong to $S_{n+1}$ for some $n\geq 1$ at time $t-$, and if $u \leq q_n$, then the new value of the field at $i$ will be in $S_n$.
When the neighbours are in $S_1$, if $u \leq q_0$, then the new value of the field at $i$ will also belong to $S_1$ and does not depend on the specific values of the neighbours, provided that they belong to $S_1$.

We associate with the dynamics an \emph{interacting particle system} which will indicate our knowledge of the coalescence of the algorithm.
However, contrary to Section \ref{Spin system associated with the dynamics}, particles can have many states, not just 0 or 1. It will therefore no longer be a spin system in the sense of \cite{Liggett1985InteractingParticleSystems}, but we still call the \emph{spin} the state of a site in a certain configuration.

The set of spin configurations is $\N^{\Z^d}$.
For any $\tau \in \R$ and $\kappa \in \N^{\Z^d}$, define a continuous-time Markov process $(\omega^{\kappa, \tau}_t)_{t \geq \tau} $ on $\N^{\Z^d}$ started at time $\tau$ with initial configuration $\omega^{\kappa, \tau}_\tau=\kappa$, and whose evolution is governed as follows.
For all $\rho=(i,t,u) \in \PP$ with $t> \tau$, letting $m=\max_{j \sim i} \omega_{t-}^{\kappa, \tau} (j)$, and
\begin{equation}\label{eq:def_s_rho}
    \mathfrak{s}_\rho:=\inf\{k : u \leq q_k\},
\end{equation}
the spin at $i$ is updated with the following rules:
\begin{enumerate}
    \item if $m=0$, put spin $\mathfrak{s}_\rho$ at site $i$;
    \item if $m \in \N^*$, put spin $\max(m-1,\mathfrak{s}_\rho)$ at site $i$;
\end{enumerate}

We now provide an intuitive interpretation of the different possibilities for a spin value.
For every $k \geq 1$, spin $k$ means that we know in a measurable way from the update marks that the value of the field belongs to $S_k$. However, spin 0 means that the value of the field belongs to $S_1$ and is independent of the initial configuration (i.e. it is a measurable function from the update marks).
With these interpretations, we can understand transition rules 1 and 2 above as translations of \eqref{maximal_coupling} and \eqref{coupling_q_L}.
If all the neighbours are in $S_m$ for $m\geq 1$ and if one encounters an update mark $\rho$ such that $m-1 \geq\mathfrak{s}_\rho$, i.e. $U_\rho\leq q_{m-1}$, then one can reach the spin $m-1$, because we know from \eqref{coupling_q_L} that the output of $\varphi$ belongs to $S_{m-1}$ (case $m>1)$, and the same holds for $m=1$ by \eqref{maximal_coupling}; in this case, we moreover know that the output of $\varphi$ is independent of the initial configuration of the dynamics.
If $m-1 < \mathfrak{s}_\rho$, the update mark cannot allow spin $m-1$; the smallest spin that is reachable is in fact $\mathfrak{s}_\rho$ (considering \eqref{coupling_q_L} with $n=\mathfrak{s}_\rho$), because we have $u\leq q_{\mathfrak{s}_\rho}$ and the neighbours are in $S_m \subset S_{\mathfrak{s}_\rho+1}$.
(The case $m=0$ is slightly different simply because there is no spin $m-1$).
The assumption $q_n \xrightarrow[n \to \infty]{} 1$ ensures that $\mathfrak{s}_\rho$ is finite.
\begin{rem}
    During an update at site $i$, if there is coalescence for all the neighbours of $i$, then the new value at site $i$ does not depend on the initial configuration, so there is coalescence at $i$ after the update. However, if there is coalescence for some neighbours but not all of them, this information will not help to get coalescence or a good level for the new value at $i$. For example, if some neighbours have coalesce and the field takes values in $S_9$ there, and the other neighbours have a value in $S_1$, the best we can hope for the new value at $i$, according to \eqref{maximal_coupling} and \eqref{coupling_q_L}, is that it belongs to $S_8$.    
    This is why a spin 0 indicates that the value of the field is the same in all the coupled dynamics \emph{and} belongs to $S_1$.
    Thus, for some updates, we forget that we have coalescence and only keep information on the level.
\end{rem}

We now give the rate function and the generator of our interacting particle systems.
Let $i \in \Z^d$ and $\omega \in \N^{\Z^d}$.
For all $k \in \N$, call $\omega^i[k]$ the configuration obtained from $\omega$ by replacing the spin at $i$ by $k$, i.e.
$\omega^i[k](j)= \cas{\omega(j)}{\text{for } j\neq i}{k}{\text{for } j=i}$.
Set $m= \max_{j \sim i} \omega(j)$.
Then, the rate function $c(\cdot, \cdot, \cdot)$ is defined by
\begin{subequations}\label{rate_omega}
\begin{align}
    c(0,i,\omega) &= q_0 \si m \leq 1, \\
    c(k, i, \omega) &= \cas{q_k}{\si k=m-1}{q_k - q_{k-1}}{\si k \geq m} \text{ for every } k\in \N^*,
\end{align}
\end{subequations}
and $c(\cdot, i, \omega)$ equals zero in other cases.
Here, $c(\cdot, i, \omega)$ is the rate at which the system jumps from $\omega$ to $\omega^i[\cdot]$.
Note that for $k=\omega(i)$, $\omega$ and $\omega^i[k]$ coincide, so there is not an actual transition.
One may observe that for every fixed configuration $\omega \in \S^{\Z^d}$ and site $i\in \Z^d$, we have $$\sum_{k \in \N} c(k, i, \omega)= 1,$$
that is the total rate at which there is an update of the spin at $i$.
The generator $\L$ of the $(\omega^{\kappa, \tau}_t)_{t \geq \tau}$ is then defined by the following relation: for all $\omega \in \S^{\Z^d}$ and every local function $f$,
\begin{equation}\label{generator_L_bis}
    \L f(\omega) = \sum_{i \in \Z^d} \sum_{k \in \N} c(k, i, \omega) (f(\omega^i[k]) - f(\omega)).
\end{equation}

These rates are consistent with the transition mechanism described above. In fact, when meeting an update mark $\rho$, a transition to $\omega^i[0]$ takes place if $U_\rho\leq q_0$, only in the case where $m=0$ or 1;
for every $k \in \N^*$, a transition to $\omega^i[k]$ takes place when $k=m-1$ and $k \geq\mathfrak{s}_\rho$, that is $U_\rho \leq q_{k}$, or when $k \geq m$ and $k=\mathfrak{s}_\rho$, that is $q_{k-1} < U_\rho \leq q_k$.

The following result formally states the intuitive interpretation of the different values for a spin in the particle system.

\begin{lemme}\label{lemma:coalescence_when_omega=0}
    Let $\tau \in \R$ and $\kappa \in (\N^*)^{\Z^d}$ (i.e. a spin configuration without $0$). Denote by $\Omega_\kappa$ the set of configurations that are \emph{compatible} with $\kappa$, i.e.
    \begin{equation*}
        \Omega_\kappa:=\{\xi \in \Omega \mid \forall i \in \Z^d, \xi_i \in S_{\kappa(i)}\}.
    \end{equation*}
    Then, for every $t\geq \tau$ and $i \in \Z^d$, we have
    $$\{\omega_t^{\kappa, \tau}(i)=0\} \as{\subset} \{\forall \xi \in \Omega_\kappa, X^{\xi,\tau}_t(i)=\varphi(\mathbf{0}, U_{\rho(i,t)}) \in S_1\}$$
    and for all $n \geq 1$
    $$\{\omega_t^{\kappa, \tau}(i) \leq n\} \as{\subset} \{\forall \xi \in \Omega_\kappa, X^{\xi,\tau}_t(i) \in S_n\}.$$
\end{lemme}

\begin{proof}
    The second inclusion is true for all $i \in \Z^d$ and $n\geq 1$ at time $t=\tau$ because by definition of $\Omega_\kappa$, the event on the right has probability 1; the first one also holds because the event on the left-hand side is empty, as $\kappa(i) \neq0$. Consider $(i,t,u) \in \PP$ for some $t> \tau$, and suppose that both properties hold for all $s\in[\tau, t)$ (this is referred to as the ``induction hypothesis''). We check that they are a.s. preserved by the update, so that they hold at time $t$. We only need to verify the inclusions for site $i$, since only the spin at $i$ can be affected by the update. Let $n \geq 1$.
    
    Let $m=\max_{j \sim i} \omega_t^{\kappa, \tau}(j)$. 
    We first consider the case $m \geq 1$; by induction hypothesis, we know that almost surely for all $\xi \in \Omega_\kappa$, $X_t^{\xi, \tau}(j) \in S_m$ for all $j \sim i$.
    The value of the spin at $i$ after the update satisfies $$\omega_t^{\kappa, \tau}(i)=\inf\{k \geq m-1 : u\leq q_k\}.$$
    It is smaller than $n$ if and only if $u \leq q_n$ and $n\geq m-1$. In this case, by \eqref{coupling_q_L} with $\eta=(X_t^{\xi, \tau}(j))_{j \sim i} \in S_m^B \subset S_{n+1}^B$, it follows that $X_t^{\xi, \tau}(i) \in S_n$. Thus, one has $\omega_t^{\kappa, \tau}(i)=0$ when $m=1$ and $u \leq q_0$. In this case, \eqref{maximal_coupling} implies that for all $\xi \in \Omega_\kappa$, one has $X_t^{\xi, \tau}(i)=\varphi(\mathbf{0},u) \in S_1$, since $(X_t^{\xi, \tau}(j))_{j \sim i} \in S_1^B$.

    Turning to the case $m=0$, the induction hypothesis tells us that almost surely for all $\xi, \xi' \in \Omega_{\kappa}$, one has $X_t^{\xi, \tau}(j)=X_t^{\xi', \tau}(j) \in S_1$ for all $j \sim i$. We only use the fact that $(X^{\xi, \tau}_t(j))_{j\sim i} \in S_1^B$ and argue as in the case $m \geq1$.
\end{proof}

If one wanted to pursue the strategy of Section \ref{Proof of thm mu_L}, one should define for all $i \in \Z^d$ a random time $\tau_i$ such that coalescence from $\tau_i$ occurs at $(i,0)$ almost surely. In terms of the interacting particle systems defined in the present Section, it would correspond to take
$$\tau_i=\sup\{ \tau <0 \mid \omega^{\kappa, \tau}_0(i)=0\},$$
with a maximal initial configuration $\kappa$, because at the initial time $\tau$, no information about the field can be derived only from the update marks. Such a maximal configuration does not exist, so we cannot exploit Proposition \ref{prop:CFTP} directly. We prove here the following stronger version. The proof still relies on a coupling-from-the-past argument.

We define the following event: there is \emph{weak coalescence} at $(i,t) \in \Z^d \times \R$ if there exists $X^*(i,t) \in S$ measurable of the update marks such that
\begin{equation}\label{hyp_prop_weakCFTP}
    \mu\left(\{\xi \in \Omega : X_t^{\xi, \tau}(i) = X^*(i,t) \}\right) \xrightarrow[\tau \to -\infty]{} 1.
\end{equation}

\begin{prop}\label{prop:strong_CFTP}
Assume that almost surely for every $i\in \Z^d$, there is weak coalescence at $(i,0)$.
Then, the random field $X^*:= (X^*(i,0))_{i \in \Z^d}$ is distributed according to $\mu$. Moreover, $X^*$ provides a realisation of $\mu$ as an FIID.
\end{prop}

\begin{proof}
    Call $\nu$ the distribution of $X^*$.
    Fix $\Lambda \Subset \Z^d$ and $\delta>0$.
    For every $i\in \Lambda$ and $\tau<0$, define the random set $$\Omega^i_\tau := \{\xi \in \Omega : X_0^{\xi, \tau}(i) = X^*(i,0)\}.$$

    By assumption, a.s. there exists $\tau_i<0$ such that a.s. for every $\tau \leq \tau_i$, we have $\mu(\Omega^i_\tau) \geq 1-\delta$.
    Let $\tau = \inf_{i \in \Lambda} \tau_i$; since $\Lambda$ is finite, a.s. $\tau > - \infty$.
    Then, define $\Omega_\tau = \bigcap_{i\in \Lambda} \Omega^i_\tau$. By a union bound, we have a.s.
    $$\mu(\Omega_\tau^c) = \mu \left( \bigcup_{i \in \Lambda} (\Omega^i_\tau)^c\right) \leq |\Lambda| \delta.$$

    Then, using the characterization of the total variation distance between two measures as 1--(maximal coupling probability), let us show that $d_{\mathrm{TV}}(\nu|_\Lambda, \mu|_\Lambda) \leq |\Lambda| \delta$. Clearly $X^*|_\Lambda$ has law $\nu|_\Lambda$.
    Let $\xi$ be a random field of law $\mu$, independent of $\PP$.
    Since $\mu$ is invariant for the dynamics, and $\xi$ has law $\mu$, it follows that $X_0^{\xi, \tau}$ is also distributed according to $\mu$. Hence, its restriction to $\Lambda$ has law $\mu|_\Lambda$.
    Therefore, $(X^*|_\Lambda, X_0^{\xi, \tau}|_\Lambda)$ yields a coupling of distributions $\nu|_\Lambda$ and $\mu|_\Lambda$.
    Moreover, $X^*|_\Lambda$ and $X_0^{\xi, \tau}|_\Lambda$ coincide on $\Omega_\tau$, so we get
    $$d_{\mathrm{TV}}(\nu|_\Lambda, \mu|_\Lambda) \leq 1-\mu(\Omega_\tau) \leq |\Lambda| \delta.$$
    Letting $\delta \to 0$, we get that $\nu|_\Lambda = \mu|_\Lambda$.
    Then, since $\Lambda$ was arbitrary, it yields $\nu=\mu$.
\end{proof}

In the following paragraphs, we will prove that there is a.s. weak coalescence at $(i,0)$ for all $i \in \Z^d$. 
By translation-invariance and because there are countably many sites, it is enough to prove a.s. weak coalescence at $(0,0)$.

\subsection{Backward exploration of the dynamics}\label{Backward exploration of the dynamics-bis}
Recall from Section \ref{Backward exploration of the dynamics} the notions of active paths and active marks from a given $\rho \in \PP$.
We introduce a new percolation model on the update marks, which is no more independent, but rather has long-range interactions, namely a \emph{Boolean percolation}.
It is defined on a graph, and the concept of an open site is replaced by that of \emph{wet site}.
This is in accordance with the terminology used in the original article by Lamperti \cite{lamperti}, in which each site is considered to be a fountain that wets all the sites in a ball of random radius around it.

We define the model on the (random) directed graph induced by the space-time positions of the update marks: the vertex set is taken to be $\PP$, and we put an edge from $\rho'$ to $\rho$ for each pair $\rho,\rho'\in \PP$ satisfying $d(\rho,\rho')=1$. Note that the edges have the opposite direction of the active paths!
For all $\rho,\rho' \in \PP$ say that $\rho'$ \emph{wets} $\rho$ if and only if $U_{\rho'} > q_{d(\rho,\rho')}$ (with convention $q_\infty=1$). Equivalently, $\rho'$ wets all the marks at distance (strictly) less than $\mathfrak{s}_{\rho'}$ from it (recall \eqref{eq:def_s_rho}), including $\rho'$ itself as soon as $\mathfrak{s}_{\rho'}>0$.
It follows that an update mark $\rho$ can only be wetted by an active mark of $A_\rho$.
The set of \emph{wet marks} is thus defined as
\begin{equation}\label{W(P)}
    W(\PP) := \{ \rho \in \PP \mid \exists \rho' \in \A_\rho, d(\rho,\rho')< \mathfrak{s}_{\rho'} \}.
\end{equation}
The update marks belonging to $W(\PP)$ are called wet, and those of $W(\PP)^c$ are called \emph{dry}.

We now define the \emph{wet active} paths, which is a more intricate notion than the one of open active paths in Section \ref{Backward exploration of the dynamics}.
For all $(i,t) \in \Z^d \times \R_-$, say that there is a wet active path from $(0,0)$ to $(i,t)$ and write $(0,0) \xrightarrow[]{w.a.} (i,t)$ if there exists an active path from $(0,0)$ to $(i,t)$ that crosses an update mark that wets $\rho(0,0)=:\rho_0$.
In particular, a wet active path is \emph{not} an active path that only crosses wet marks. It should be seen as a path of (backward) exploration which reveals the fact that $\rho_0$ is wet.
It is sufficient for our purpose to define the wet active paths starting from $(0,0)$ only, although the definition could naturally be extended to other space-time positions.

We authorise two new values for the spin of a site, namely $-\infty$ and $+\infty$, and let $\S=\N\cup\{\pm \infty\}$. For every $t \leq 0$, define $\hat \sigma_t \in \S^{\Z^d}$ by
\begin{equation}\label{def_hat_sigma}
    \forall i \in \Z^d, \hat \sigma_t (i) = \cas{-\infty}{\si (0,0) \xrightarrow[]{w.a.} (i,t)}{d(\rho_0, \rho(i,t))}{\text{ otherwise}}.
\end{equation}
Note that if there is no active path from $(0,0)$ to $(i,t)$, the definition of the distance $d$ (in Section \ref{Backward exploration of the dynamics}) implies that $\hat \sigma_t (i) = + \infty$.
It is clear that the events of $\rho_0$ being wet and the appearance of a spin $-\infty$ in the process $(\hat \sigma_t)_{t \leq 0}$ coincide.
Moreover, since the distance to $\rho_0$ is non-decreasing along every active path starting from $(0,0)$, it follows that if $\rho_0$ is dry, then the process $(\hat \sigma_t)_{t \leq 0}$ eventually exceed any $\kappa \in \N^{\Z^d}$.

We prove an analogue of Lemma \ref{couplage_disjoint} for $(\omega^{\kappa, \tau}_t)_{t \geq \tau}$ and $(\hat \sigma_t)_{t \leq 0}$. Denote by $\leq$ the natural partial order on $\S^{\Z^d}$, defined by
$$\omega \leq \sigma \ssi \forall i \in \Z^d, \omega(i) \leq \sigma(i).$$

\begin{lemme}\label{couplage_monotone}
    Let $\kappa \in \N^{\Z^d}$. For all $\tau <0$ it holds that
    $$\{\omega_0^{\kappa, \tau}(0)=0\} \as{\subset} \{\forall t \in [\tau, 0], \omega^{\kappa, \tau}_t \leq \hat \sigma_t\}.$$
    In particular, $\{\omega_0^{\kappa, \tau}(0)=0\} \as{\subset} \{\hat \sigma_\tau \geq \kappa\}$.
\end{lemme}
\begin{proof}
    We first show that on the event that $\omega^{\kappa, \tau}_0(0)=0$, a.s. no update mark between times $\tau$ and $0$ wets $\rho_0$.
    Since the spin in the process $(\omega^{\tau,\kappa}_t)_{t \geq \tau}$ can only increase by at most 1 along the active paths (in the forward sense, an update can decrease the spin by at most 1), $\omega^{\tau, \kappa}_0(0)=0$ implies that for all $\rho=(i,t,u) \in \A_{\rho_0}$ such that $t\geq \tau$, one has $\omega^{\kappa, \tau}_t(i) \leq d(\rho_0,\rho)$.
    Arguing by contradiction, suppose that there exists $\rho=(i,t,u) \in \PP$ with $t\in[\tau,0]$ such that $d(\rho_0,\rho)<\mathfrak{s}_\rho$. The update encoded by $\rho$ leads to $\omega^{\kappa, \tau}_t(i) \geq \mathfrak{s}_\rho >d(\rho_0,\rho)$, which contradicts $\omega^{\kappa, \tau}_0(0)=0$.

    We can now prove the inequality $\omega^{\kappa, \tau}_t \leq \hat \sigma_t$ for all $t\in[\tau,0]$. The outline is the same as in the proof of Lemma \ref{couplage_disjoint}.
    For $t=0$, the property is true because $\hat \sigma_0$ a.s. equals $+\infty$ everywhere except at $0$, where we have $\hat \sigma_0(0) =0$.
    Therefore, on the event $\omega^{\kappa,\tau}_0(0) = 0 $, one has $\omega^{\kappa,\tau}_0(i) \leq \hat \sigma_0(i)$ for every $i \in \Z^d$.

    Now, we check that the property is a.s. preserved by the (backward) updates.
    Suppose that $\rho=(i,t,u) \in \PP$ for some $t \in [\tau,0]$ and that $\omega^{\kappa, \tau}_t \leq \hat \sigma_t$; let us show that $\omega^{\kappa, \tau}_{t-} \leq \hat \sigma_{t-}$.    
    First, since we have a mark at $(i,t)$, we always have $\hat \sigma_{t-}(i) = + \infty$ because there is no active path from $(0,0)$ to $(i,t)$, so $\omega^{\kappa, \tau}_{t-}(i) \leq \hat \sigma_{t-}(i)$.
    The update induced by $(i,t,u)$ can only affect the spin at $i$ in $\omega^{\kappa,\tau}_t$, so $\omega^{\kappa, \tau}_{t-}(j) =\omega^{\kappa, \tau}_t(j) $ for every $j\neq i$.
    If $j$ is not a neighbour of $i$, we also have $\hat \sigma_{t-}(j) =\hat \sigma_t(j)$, so the inequality $\omega^{\kappa, \tau}_{t-}(j) \leq \hat \sigma_{t-}(j)$ holds.
    Therefore, we only need to handle the case where $j$ is a neighbour of $i$.
    If $j \sim i$, we have 
    $$\hat \sigma_{t-}(j) = \cas{-\infty}{\si \rho \text{ wets } \rho_0}{\inf\{\hat \sigma_t(i)+1, \hat \sigma_t(j)\}}{\text{ otherwise}}.$$    
    Indeed, if $\rho$ wets $\rho_0$, it yields a wet active path from $(0,0)$ to $(j, t-)$, so $\hat \sigma_{t-}(j)=-\infty$.
    Otherwise, if $\hat \sigma_t(i)=-\infty$ or $\hat \sigma_t(j)=-\infty$, it also yields a wet active path from $(0,0)$ to $(j, t-)$, although the fact that it is wet is not due to $\rho$, so $\hat \sigma_{t-}(j)=-\infty$.
    Assume now that $\hat \sigma_t(i) \geq 0$ and $\hat \sigma_t(j) \geq 0$. The active paths, if any, from $(0,0)$ to $(j,t)$ continue to $(j,t-)$, because there is no mark at $(j,t)$. Their minimal length is given by $\hat \sigma_t(j)$ (which is $+\infty$ if there is none). On the other hands, the active paths from $(0,0)$ to $(i,t)$ and the mark $\rho$ create active paths from $(0,0)$ to $(j,t-)$, whose minimal length is therefore $\hat \sigma_t(i)+1$. Then, the value taken by $\hat \sigma_{t-}(j)$ is the minimal length of all these active paths.
    
    We have shown in the beginning of the proof that ``$\rho$ wets $\rho_0$'' and $\omega^{\kappa, \tau}_0(0)=0$ are disjoints events, so we have $\hat \sigma_{t-}(j) = \inf\{\hat \sigma_t(i)+1, \hat \sigma_t(j)\}$.
    Furthermore, according to the transition mechanism of the particle system (see Section \ref{Interacting particle system associated with the dynamics}), we have $\omega^{\kappa, \tau}_t(i) \geq \max_{j \sim i} \omega^{\kappa, \tau}_{t-}(j) -1$.
    Hence, we get $$\hat \sigma_t(i)+1 \geq \omega^{\kappa, \tau}_t(i)+1 \geq \omega^{\kappa, \tau}_{t-}(j) \et \hat \sigma_t(j) \geq \omega^{\kappa, \tau}_t(j) = \omega^{\kappa, \tau}_{t-}(j).$$
    In both statements, the first inequality comes from the ``induction hypothesis'' $\omega^{\kappa, \tau}_t \leq \hat \sigma_t$.
    It follows that $\hat \sigma_{t-}(j) = \inf\{\hat \sigma_t(i)+1, \hat \sigma_t(j)\} \geq \omega^{\kappa, \tau}_{t-}(j)$.
\end{proof}

\subsection{Duality between the particle system and the backward exploration}\label{Duality between the particle system and the backward exploration}
We now exploit the duality to get the reverse inclusion.
The Definition \ref{def:duality} of duality holds for the class of particle systems studied here. However, since the state space is no longer $\{0, 1\}^{\Z^d}$, there is no good way to see a spin configuration as a subset of $\Z^d$.
Therefore, instead of the indicator function of the event that the two configurations do not intersect, the appropriate duality function is
\begin{equation*}
    H(\omega, \sigma) = \ind{\omega \leq \sigma}.
\end{equation*}

Let $\Sigma \subset \S^{\Z^d}$ be the set of configurations with finitely many spins different from $+\infty$.
For every $\sigma \in \Sigma$, $i\in \Z^d$ and $k \in \S$, let $\sigma^{i,B}[k]$ be the configuration derived from $\sigma$ by putting spin $+\infty$ to site $i$ and the minimum between $k$ and the current spin to every neighbour of $i$, i.e.
\begin{equation}\label{transition_sigma}
    \sigma^{i,B}[k](j) = \begin{cases}
    +\infty & \text{ if } j=i\\
    \min(\sigma(j),k) & \text{ if } j\sim i\\
    \sigma(j) & \text{ otherwise}
    \end{cases}.
\end{equation}

Then, introduce a generator $\M$ whose action is defined over every function $f:\Sigma\to \R$ by
\begin{equation}\label{generator_M_bis}
    \M f(\sigma) = \sum_{i \in \Z^d} \Big[ q_{\sigma(i)} f\left(\sigma^{i,B}[\sigma(i)+1]\right) + (1 - q_{\sigma(i)}) f\left(\sigma^{i,B}[-\infty]\right) - f(\sigma) \Big]
\end{equation}
with conventions $q_{-\infty}=0$, $q_{+\infty}=1$ and $\pm \infty +1 = \pm \infty$,
and observe that the $i-$th term is just $f\left(\sigma^{i,B}[-\infty]\right) - f(\sigma)$ when $\sigma(i)=- \infty$ and $0$ when $\sigma(i)=+\infty$, because $\sigma^{i, B}[+\infty]=\sigma$, so in this last case there is not an actual transition.
Note that this expression is well defined for every $\sigma \in \Sigma$ without restriction on the function $f$, since the sum has finitely many non-zero terms.

We first give an intuitive interpretation of this generator in terms of the graphical representation.
Consider a backward process $(\sigma_t)_{t \leq 0}$ with state space $\Sigma$ and initial configuration $+\infty$ everywhere except at the origin, where we set $\sigma_0(0)=0$.
The process is left-continuous and for all $(i,t,u) \in \PP$ with $t < 0$, if $\sigma(i):=\sigma_{t+}(i)$, we update the system as follows:
\begin{itemize}
    \item if $u\leq q_{\sigma(i)}$, transition to $\sigma^{i,B}[\sigma(i)+1]$ 
    \item if $u>q_{\sigma(i)}$, transition to $\sigma^{i,B}[-\infty]$.
\end{itemize}
It should be clear that the generator of $(\sigma_t)_{t \leq 0}$ is $\M$.

\begin{rem}
    The action of $\M$ is only defined on functions of $\Sigma$. We justify here why $\sigma_t$ stays in $\Sigma$ for all $t<0$. This is the same argument of Harris, already mentioned in the end of Section \ref{Glauber dynamics}, used to construct the graphical representation of the Glauber dynamics. Fix $t_0<0$ and consider the following random subgraph of $\Z^d$: draw an edge between $i$ and $j$ if and only if there is an update mark carried by $i$ or $j$ between times $0$ and $t_0$. Harris proved that for $|t_0|$ small enough, this subgraph $\P-$a.s. only contains finite connected components. A site can have a spin different from $+\infty$ in $\sigma_{t_0}$ only if it is in the same connected component as a site which has a spin different from $+\infty$ in $\sigma_0$; since we start with a configuration $\sigma_0 \in \Sigma$ and all the connected components are finite, it follows that at time $t_0$ only finitely many sites have a spin different from $+\infty$. We then deduce recursively that $\sigma_t \in \Sigma$ for all $t$. 
\end{rem}

The possible values of a spin in $\sigma_t$ represent, as in Section \ref{Dual process}, what is needed at time $t$ in the dynamics to get almost sure coalescence at $(0,0)$ from a certain negative time.
If $\sigma_t(i) = +\infty$ for some $t\leq 0$ and $i\in \Z^d$, it means that no information on the value of the field at $(i,t)$ is needed.
If $\sigma_t(i)=k$ for some $k \geq 1$, it means that we need to know in a measurable way from the update marks that the value of the field in the dynamics belongs to $S_k$; $\sigma_t(i)=0$ means that we also need that the value of the field is in $S_1$, and moreover that it does not depend on the initial configuration of the dynamics.
Finally, $\sigma_t(i) = -\infty$ means that we cannot have almost sure coalescence at $(0,0)$.
Transition rules are then understandable in the following way.
Consider an update mark $(i,t,u)$ with $t<0$ and $\sigma:=\sigma_{t+}$ the configuration before the update. In order to get the level we want for the dynamics at $(i,t)$, encoded by $\sigma(i)$, we need to know that the field values on the neighbourhood of $i$ at time $t-$ belong to $S_{\sigma(i) +1}$.
Furthermore, we need a ``good'' update mark, that is here $u \leq q_{\sigma(i)}$, so that the update function $\varphi$ returns a value in $S_{\sigma(i)}$ (see \eqref{maximal_coupling} and \eqref{coupling_q_L}).
Thus, if $u \leq q_{\sigma(i)}$, jump from configuration $\sigma$ to $\sigma^{i,B}[\sigma(i)+1]$: 
the neighbours of $i$ take spin $\sigma(i)+1$ or keep their spins if it is smaller, because a smaller spin means that a more restrictive level is already required.
If $u>q_{\sigma(i)}$, the update mark is not good enough to reach the target level.
In this case, the coalescence cannot be achieved almost surely at $(0,0)$, so spins $-\infty$ invade the process (jump to $\sigma^{i,B}[-\infty]$).

We now formally prove the duality.

\begin{lemme}\label{prop_duality_bis}
The particle systems generated by $\L$ and $\M$ are dual with respect to $H(\omega, \sigma) = \ind{\omega \leq \sigma}$.
\end{lemme}

\begin{proof}
The proof is very similar to the one of Lemma \ref{prop_duality}, and relies on the same result about generators (Theorem 3.42 in \cite{liggett2010continuous}).
In particular, we only prove that $$\L H(\cdot,\sigma)(\omega)=\M H(\omega, \cdot)(\sigma)$$ for all $\omega \in \N^{\Z^d}$, $\sigma \in \Sigma$, where the left-hand side is well-defined because $\sigma$ takes spin $+\infty$ almost everywhere, so $H(\cdot, \sigma)$ is a local function.

Fix $i \in \Z^d$, $\omega \in \N^{\Z^d}$, $\sigma \in \Sigma$, and let $m=\max_{j \sim i} \omega(j)$.
For all $k \in \N$,
set $A_k = H(\omega^i[k], \sigma) - H(\omega, \sigma)$ and $B_k = H(\omega, \sigma^{i,B}[k]) - H(\omega, \sigma)$, and similarly set $B_{\pm\infty}=H(\omega, \sigma^{i,B}[\pm\infty]) - H(\omega, \sigma)$. Note that one always has $B_{+\infty}=0$.
It is sufficient to show that
\begin{equation}
    \label{duality_generators}
    \sum_{k \in \N} c(k, i, \omega) A_k= q_{\sigma(i)} B_{\sigma(i)+1} + (1-q_{\sigma(i)}) B_{-\infty},
\end{equation}
where $c(k, i, \omega)$ is given by \eqref{rate_omega}, since then we can sum over $i \in \Z^d$ to get the result.

First, note that if there is some $j \neq i$ such that $\sigma(j) < \omega(j)$, then for every $k \in \N$, $A_k = B_k=B_{-\infty}=0$.
Indeed, in this case the condition $\omega \leq  \sigma$
fails at site $j$ so $H(\omega, \sigma) = 0$. Then, $\omega^i[k](j)=\omega(j) > \sigma(j)$ so we also have $H(\omega^i[k], \sigma) = 0$, thus $A_k=0$.
The spin at $j$ may differ between $\sigma$ and $\sigma^{i,B}[k]$, but $\sigma^{i,B}[k](j) \leq \sigma(j) < \omega(j)$ because $j\neq i$, so we still have $H(\omega, \sigma^{i,B}[k])=0$, thus $B_k=0$ (this sentence is also valid for $k=-\infty$).
In this case, \eqref{duality_generators} holds obviously.

Consequently, in the rest of the proof, we handle the case:
\begin{equation*}
    \forall j \neq i, \ \omega(j) \leq \sigma(j).
\end{equation*}
Under this assumption, one can compute $A_k$ and $B_k$ for all $k \in \N$ (including $k=-\infty$ for $B_k$):
\begin{itemize}
    \item $A_k = \cas{1}{\si  k\leq \sigma(i) \et \omega(i)>\sigma(i)}{-1}{\si k> \sigma(i) \et \omega(i) \leq \sigma(i) }$;
    \item $B_k=\cas{1}{\si m \leq k \et  \omega(i) > \sigma(i)}{-1}{\si m > k \et \omega(i) \leq \sigma(i)} \hspace{1 cm} \et \hspace{1 cm} B_{-\infty}=-\ind{\omega(i) \leq \sigma(i)}.$
\end{itemize}
Now, verify the equality in \eqref{duality_generators}, considering separately the cases
$\omega(i) \leq \sigma(i)$ and $\omega(i) > \sigma(i)$.
We will refer to the left-hand side of \eqref{duality_generators} by the generic symbol $L$ and to the right-hand side by $M$.

\begin{enumerate} 
    \item $\boldsymbol{\omega(i) \leq \sigma(i):}$ In this case, we have for all $k\in \S$, $A_k = - \ind{k> \sigma(i)} $ and $B_k=-\ind{m>k}$, and $B_{-\infty}=-1$.
    When $\sigma(i) = +\infty$, one has $M= B_{+\infty}=0$, and $L= 0$ because all the $A_k$ are null.    
    When $\sigma(i)< + \infty$, one has $-M= -q_{\sigma(i)} B_{\sigma(i)+1} + 1-q_{\sigma(i)} $, so
    $$-M=\cas{1-q_{\sigma(i)}}{\si m\leq \sigma(i)+1}{1}{\si m>\sigma(i)+1}.$$
    On the other hand, $$-L= \sum_{k= \sigma(i)+1}^{+\infty} c(k, i, \omega).$$
    Now, $c(k, i, \omega) = 0$ if $k<m-1$, so if $m\leq \sigma(i)+1$, the sum starts at $\sigma(i)+1$, but if $m> \sigma(i)+1$, one can start the sum at $m-1$.
    In the first case, we have $-L= q_{\sigma(i)+1} - q_{\sigma(i)} + q_{\sigma(i)+2} - q_{\sigma(i)+1} + \dots = 1-q_{\sigma(i)}$
    and in the second case we have
    $-L= q_m + q_{m+1} - q_m + q_{m+2} - q_{m+1}+ \dots = 1$.
    Hence $L=M$.

    \item $\boldsymbol{\omega(i) > \sigma(i):}$
    Here, $A_k=\ind{k \leq \sigma(i)}$, $B_k = \ind{m \leq k}$ and $B_{-\infty}=0$, and note that the case $\sigma(i) = + \infty$ is excluded.
    Thus, $M=q_{\sigma(i)} \ind{m \leq \sigma(i)+1}$.
    Then,
    $$L= \sum_{k=0}^{\sigma(i)} c(k, i, \omega) = \sum_{k=m-1}^{\sigma(i)} c(k, i, \omega) = \cas{0}{\si m>\sigma(i)+1}{q_{\sigma(i)}}{\si m \leq \sigma(i) +1},$$
    coinciding with $M$, where the $q_{\sigma(i)}$ again arises from a telescopic sum.\qedhere
\end{enumerate} \end{proof}

We can now deduce from the previous lemma the duality between the particle systems introduced in Section \ref{Interacting particle system associated with the dynamics} and the backward exploration.

\begin{prop}
    For all $\tau \in \R$ and $\kappa \in \N^{\Z^d}$, the processes $(\omega^{\kappa, \tau}_t)_{t \geq \tau}$ and $(\hat \sigma_{\tau-t})_{t \geq \tau}$ are dual with respect to $H(\omega, \sigma) = \ind{\omega \leq \sigma}$.
\end{prop}
\begin{proof}
    Fix $\tau \in \R$. According to Lemma \ref{prop_duality_bis}, one only has to verify that the generator of $(\hat \sigma_{\tau-t})_{t \geq \tau}$ is $\M$. Recall that this process has been defined through the graphical representation, and note that it is left-continuous.
    For some $t\geq \tau$, suppose that there is an update mark $\rho=(i, \tau-t, u) \in \PP$. It affects the process as follows:
    $$\hat \sigma_{(\tau-t)-}(j)= \begin{cases}
    +\infty & \text{ if } j=i\\
    \min(\hat \sigma_{\tau-t}(j),k) & \text{ if } j\sim i\\
    \hat \sigma_{\tau-t}(j) & \text{ otherwise}
    \end{cases} \text{ i.e. } \hat \sigma_{(\tau-t)-}=\hat \sigma_{\tau-t}^{i,B}[k],$$
    where $k=-\infty$ if $u > q_{d(\rho_0,\rho)}$ (i.e. if $\rho$ wets $\rho_0$) and $\hat \sigma_{\tau-t}(i)+1$ otherwise.
    This is a consequence of \eqref{def_hat_sigma}, already observed in the proof of Lemma \ref{couplage_monotone}.
    The value of $\hat \sigma_{\tau-t}(i)$ can either be $d(\rho_0, \rho)$ or $-\infty$.
    If it is $-\infty$, one has $\hat \sigma_{(\tau-t)-}= \hat \sigma_{\tau-t}^{i, B}[-\infty]$.
    Otherwise, the condition $u>q_{d(\rho_0,\rho)}$ is equivalent to $u>q_{\hat \sigma_{\tau-t}(i)}$, and in this case one has $\hat \sigma_{(\tau-t)-}= \hat \sigma_{\tau-t}^{i, B}[-\infty]$; if $u \leq q_{\hat \sigma_{\tau-t}(i)}$, one has $\hat \sigma_{(\tau-t)-}= \hat \sigma_{\tau-t}^{i, B}[\hat \sigma_{\tau-t}(i)+1]$.
    
    Therefore, for every $i \in \Z^d$, transitions of the form $\sigma$ to $\sigma^{i, B}[-\infty]$ appear at rate $1-q_{\sigma(i)}$, and transitions of the form $\sigma$ to $\sigma^{i, B}[\sigma(i)+1]$ appear at rate $q_{\sigma(i)}$. These are the rates prescribed by $\M$.
\end{proof}

The duality relation \eqref{def_duality} (with function $H(\omega, \sigma) = \ind{\omega \leq \sigma}$) applied at time $t=0$ for processes $(\omega_t^{\kappa, \tau})_{t \geq \tau}$ and $(\hat \sigma_{\tau-t})_{t \geq \tau}$ gives
$$\P(\omega_0^{\kappa, \tau} (0) = 0) = \P(\kappa \leq \hat \sigma_\tau),$$
because $\omega^{\kappa, \tau}_\tau=\kappa$ and $\hat \sigma_0$ has spin $+\infty$ everywhere except at the origin, where it is 0.
Coupled with the inclusion obtained in Lemma \ref{couplage_monotone}, it gives
\begin{equation}\label{eq:duality_bilan_bis}
    \{\omega^{\kappa, \tau}_0(0)=0\}\as{=}\{\hat \sigma_\tau \geq \kappa\}.
\end{equation}

Then, if $\rho_0$ is dry, for every $\kappa \in \N^{\Z^d}$ one can choose $\tau<0$ such that $\hat \sigma_\tau \geq \kappa$. By \eqref{eq:duality_bilan_bis} and Lemma \ref{lemma:coalescence_when_omega=0}, it follows that we have coalescence at $(0,0)$ of all the processes $(X^{\xi, \tau}_t)_{t \geq \tau}$ with initial condition $\xi\in \Omega_\kappa$.
By ``letting $\kappa$ go to infinity'', we will then obtain a.s. weak coalescence at $(0,0)$ conditionally on the event that $\rho_0$ is dry.
We then explain how to proceed when $\rho_0$ is wet.

\subsection{Proof of Theorem \ref{thm_general}}

In this Section, we conclude the proof of Theorem \ref{thm_general}. We start by two lemmas, the first one formalising the discussion at the end of the previous paragraph, that claims that everything goes well for dry marks, and the second one stating that wet marks do not percolate, which turns out to be sufficient to get a coding for $\mu$. Note that hypotheses \eqref{link_L0_p0} and \eqref{link_Ln_pn} have been used to define the update function of the dynamics. The two others hypotheses, \eqref{condition_qn} and \eqref{condition_Ln}, will be used, respectively, in the proofs of Lemma \ref{cutset} and Lemma \ref{lemma:X_rho_dry_marks} below.

\begin{lemme}\label{lemma:X_rho_dry_marks}
    For every $\rho=(i,t,u) \in \PP$, one has
    $$\{\rho \text{ is dry} \} \as \subset\{\forall \tau <t, \ X^{\mathbf{0},\tau}_t(i)=\varphi(\mathbf{0}, u) \}$$ and
    $$\{\rho \text{ is dry} \} \as{\subset} \left\{\mu\left(\{\xi \in \Omega :  X_t^{\xi, \tau}(i)=\varphi(\mathbf{0},u)\right) \xrightarrow[\tau \to -\infty]{}1 \right\}.$$
\end{lemme}
\begin{proof}
    It is sufficient to prove the statements for $\rho_0=\rho(0,0)$, by space and time translation-invariance.
    Consider the process $(\hat \sigma_t)_{t \leq 0}$ introduced in Section \ref{Backward exploration of the dynamics-bis}, and let $t_0$ be the time of the update mark $\rho_0$, i.e. the only negative number $t_0$ such that $\rho_0=(0,t_0,U_{\rho_0})$.
    On the event of $\rho_0$ being dry, by definition of $(\hat \sigma_t)_{t \leq 0}$, one has $\hat \sigma_t \geq 1$ for all $t < t_0$ (the update encoded by $\rho_0$ gives spin $1$ to the neighbours of the origin, and spin $+\infty$ everywhere else; in the rest of the evolution of the process, the spin at any site cannot take a value smaller than 1, because $\rho_0$ is dry).
    The duality proved in Section \ref{Duality between the particle system and the backward exploration} then implies (see \eqref{eq:duality_bilan_bis}) that $\omega^{\kappa, \tau}_0(0)=0$ for all $\tau < t_0$, with $\kappa \equiv1$. By Lemma \ref{lemma:coalescence_when_omega=0}, we get in particular that a.s.
    $$\forall \tau <t_0, \ X_0^{\mathbf{0}, \tau}(0)= \varphi(\mathbf{0}, U_{\rho_0}),$$
    because $\mathbf{0} \in \Omega_\kappa$.
    We now prove the second statement.
    Define a sequence $(\kappa_n)_{n \geq 1}$ of spin configurations by $\kappa_n(i):=\max(n, \Vert i\Vert_1)$ for all $i \in \Z^d$ and $n \geq 1$.
    On the event where $\rho_0$ is dry, it holds that the process $(\hat \sigma_t)_{t \leq 0}$ eventually exceed every $\kappa \in \N^{\Z^d}$, since there is never the appearance of a spin $-\infty$. It follows that for every $n \geq 1$, there exists $\tau_n>-\infty$ such that $\hat \sigma_{\tau_n} \geq \kappa_n$. Therefore, as a consequence of the duality, we get that $\omega^{\kappa_n, \tau_n}_0(0)=0$, see \eqref{eq:duality_bilan_bis}.
    By Lemma \ref{lemma:coalescence_when_omega=0}, it then follows that for every $\xi \in \Omega_{\kappa_n}$, we have a.s.
    $$X^{\xi, \tau_n}_0(0)=\varphi(\mathbf{0}, U_{\rho_0}).$$
    Thus, we get that for every $\tau \leq \tau_n$, $$\Omega_{\kappa_n} \as\subset\Omega^0_\tau := \{\xi \in \Omega : X_0^{\xi, \tau}(0) = \varphi(\mathbf{0},U_{\rho_0})\}.$$
By hypothesis \eqref{condition_Ln} of Theorem \ref{thm_general}, one has $\mu(\Omega_{\kappa_n}) \xrightarrow[n \to +\infty]{}1$, so it gives $\mu(\Omega_\tau^0) \xrightarrow[\tau \to -\infty]{}1$. 
\end{proof}

In words, Lemma \ref{lemma:X_rho_dry_marks} says that if $\rho_0$ is dry, then we have weak coalescence at $(0,0)$, and it identify the coalescing value as $\varphi(\mathbf{0}, U_{\rho_0})$ (note that this is clearly measurable of the update marks). Unfortunately, it is not true that $\rho_0$ is almost surely dry, so it is not sufficient to apply Proposition \ref{prop:strong_CFTP} and get a coding for $\mu$,.

To deal with the case of $\rho_0$ being wet, we prove that the Boolean percolation on the update marks is in some sense subcritical, and will see in the proof of Theorem \ref{thm_general} that this is sufficient for our purpose.
A finite subset of $\A_{\rho_0}$ is said to be a \emph{cutset} (for $(0,0)$) if every infinite active path starting from $(0,0)$ passes through a mark of the subset (e.g. the smallest cutset for $(0,0)$ is $\{\rho_0\}$).
Let $\mathcal{T}:=\sigma((T^i_n)_{i \in \Z^d, n \in \Z})$, that is the $\sigma-$algebra generated by the space-time positions of the update marks. 
\begin{lemme}\label{cutset}
    There almost surely exists a dry cutset for $(0,0)$, that is a cutset included in $W(\PP)^c$.

    Moreover, there exists $\lambda>0$ such that for all $n\geq 0$,
    $$\P(\exists \text{ a wet path of length $n$ starting from }(0,0) \mid \mathcal{T}) \leq e^{-\lambda n}.$$
\end{lemme}
In this statement, a wet path should be understood as an active path that only crosses wet marks. Equivalently, a wet path of length $n$ starting from $(0,0)$ is made of $n$ consecutive update marks that are all wet, the first being $\rho_0$.

This result is a direct consequence of Theorem 2.4 in \cite{Faipeur}, that we reproduce here.
\begin{thm}[Theorem 2.4 in \cite{Faipeur}]
    Let $G$ be an infinite directed graph of bounded degree. Let $\Delta_+$ (resp. $\Delta_-$) be the maximal out-degree (resp. in-degree) of a vertex, that is the number of edges emanating from (resp. pointing at) this vertex, and assume that they both are at least 2. We consider the Boolean percolation on $G$ with radii distributed according to $(q_n)_{n \geq0}$, that is, we consider the wet set 
    $$W=\bigcup_{v \in G} B_{R_v}(v),$$
    where $B_r(v):=\{v' \in G : d(v,v')<r\}$ and the $(R_v)_{v \in G}$ are i.i.d. $\N-$valued random variable satisfying $\mathbb P(R_v \leq n)=q_n$ for all $n\geq0$. Although $G$ is directed, we only consider the weak connectivity, that is the connected components of $G$ or $W$ are those of the underlying undirected graph. Let $\rho \in G$ and denote by $W_\rho$ the connected component of $\rho$ in $W$ (with convention $W_\rho=\emptyset$ if $\rho \notin W$). 
    \begin{enumerate}
        \item If $(q_n)_{n\geq0}$ satisfies
        \begin{equation*}
            4(\Delta_+ + \Delta_-)\sum_{n \geq 0} \Delta_+^n \Delta_-^n (1-q_n) < 1,
        \end{equation*}
    then almost surely all connected components of $W$ are finite.
        \item If moreover we have for some $t>0$
        \begin{equation*}
            \sum_{n \geq 0} \Delta_-^n e^{t \Delta_+^n} (1-q_n) < \infty,
        \end{equation*}
    then there exists $\lambda>0$, depending only on $\Delta_+$, $\Delta_-$ and the $(q_n)_{n \geq 0}$, such that for every $n \in \N$, $$\P(|W_\rho| > n) \leq e^{-\lambda n}.$$
    \end{enumerate}
\end{thm}

\begin{proof}[Proof of Lemma \ref{cutset}]
    When conditioning on $\mathcal{T}$, $\A_{\rho_0}$ can be seen as a directed graph (already defined in the beginning of Section \ref{Backward exploration of the dynamics-bis}): the vertex set is $\A_{\rho_0}$, and we put an oriented edge from $\rho'$ to $\rho$ for each pair $\rho,\rho' \in \A_{\rho_0}$ such that $d(\rho, \rho')=1$. Therefore, $\A_{\rho_0}$ is a directed graph of bounded degree, where all the vertices have in-degree exactly $|B|$ and out-degree bounded by $|B|$.

    Therefore, we can apply the above Theorem with $\Delta_+=\Delta_-=|B|$, and we obtain exactly the hypothesis \eqref{condition_qn} as a sufficient condition for subcriticality and exponential decay of the Boolean percolation.
\end{proof}

We can now complete the proof of Theorem \ref{thm_general}.
\begin{proof}[Proof of Theorem \ref{thm_general}]
    The first step of the proof is to show that there is a.s. weak coalescence at $(0,0)$, regardless of the fact that $\rho_0$ is wet or dry.
    The key point is that even if $\rho_0$ is wet, one can find a dry cutset on which all the field values in the dynamics are independent of the initial condition. Then, using the update function, we can recursively compute everything after this dry cutset to determine the value at $(0,0)$.

    Let $\C$ be the first dry cutset for $(0,0)$, which is obtained by exploring all the active paths starting from $\rho$ and, for each, take the first dry mark encountered. The fact that $\C$ exists a.s. is a consequence of Lemma \ref{cutset}.
    Then, we consider a starting time $\tau_0$ preceding all the times carrying a mark of $\C$, i.e. $\tau_0$ such that every $\rho=(i,t,u) \in \C$ satisfies $t\in(\tau_0,0)$ (this is possible to do so because $\C$ is finite).
    Because they are all dry, according to Lemma \ref{lemma:X_rho_dry_marks}, the marks $(i,t,u) \in \C$ verify a.s.
    $$X^{\mathbf{0},\tau_0}_t(i)=\varphi(\mathbf{0}, u)$$
    and there is a.s. weak coalescence at $(i,t)$.
    It follows that there is a.s. weak coalescence at $(0,0)$.
    Indeed, let $X_0^*:= X^{\mathbf{0}, \tau_0}_0(0)$, and observe that $X_0^*$ is a measurable function of the update marks (it actually does not depend on $\tau_0$, as soon as it is chosen to be smaller that the infimum of times carrying an update mark of $\C$).
    For any $\tau<\tau_0$, if the dynamics started at time $\tau$ with initial configuration $\xi \in \Omega$ and the one started at $\tau_0$ with initial configuration $\mathbf{0}$ coincide at every mark of $\C$, then they will also coincide at $(0,0)$ almost surely.
    Thus, we have for all $\tau<\tau_0$, a.s.
    $$\{ \xi \in \Omega, X^{\xi,\tau}_0(0) = X_0^*)\} \supset \bigcap_{(i,t,u) \in \C} \{ \xi \in \Omega, X^{\xi,\tau}_t(i) = \varphi(\mathbf{0},u) \}.$$
    On the event of existence of $\C$, we get by Lemma \ref{lemma:X_rho_dry_marks} and finiteness of $\C$ that the $\mu$--measure of the right-hand side tends to 1 when $\tau$ goes to $-\infty$. Hence $\mu\left(\{ \xi \in \Omega, X^{\xi,\tau}_0(0) = X_0^*\}\right) \xrightarrow[\tau \to -\infty]{}1$, so we get weak coalescence at $(0,0)$.
    
    Everything has been done in a translation-invariant manner, so for all $i \in \Z^d$, we get a similar construction of $X_i^*$, and a.s. weak coalescence at $(i,0)$. Since there are only countably many sites in $\Z^d$, and according to Proposition \ref{prop:strong_CFTP}, it yields a realisation of $\mu$, which has been constructed as a factor of the update marks. This proves the first part of Theorem \ref{thm_general}, that is, $\mu$ is an FIID. It remains to show that $\mu$ can be well-approximated by finitely dependent fields.

    Let $l \geq 1$. Our goal is to construct a field $Y^*=(Y^*_i)_{i \in \Z^d}$ that is $l-$dependent and close to $X^*=(X_i^*)_{i \in \Z^d}$. The strategy is simply to truncate the factor map of $X^*$, to obtain $Y^*$ as an FIID with coding radius bounded by $\lfloor l/2 \rfloor$, so that it is $l-$dependent. We then use a union bound over all the sites of any finite box of $\Z^d$ to get the result.

    We do the construction only for $Y_0^*$, in a translation-invariant manner.
    Let $\C_l$ be the cutset (for $(0,0)$) of update marks at distance $\lfloor l/2 \rfloor$ from $\rho_0$. For all $\rho \in \C_l$, let $Y_\rho:= \varphi( \mathbf{0}, U_\rho)$.
    Then, define $Y_\rho$ recursively for every update mark $\rho$ such that $d(\rho_0,\rho)<\lfloor l/2 \rfloor$.
    Suppose that all the $Y_{\rho'}$ for $\rho'$ between $\rho$ and $\C_l$ have been defined\footnote{Note that it is possible that such an update mark $\rho'$ satisfies $d(\rho_0, \rho') \leq d(\rho_0,\rho) $, due to an active path from $(0,0)$ to $\rho'$ that does not passes through $\rho$. Thus, it is not possible to define the $Y_\rho$ by decreasing order of the distance from $\rho_0$.}; denote by $i$ the location (in $\Z^d$) of $\rho$; let
    $$Y_\rho:=\varphi(\eta, U_\rho)$$
    where for every $j \in B$, $\eta_j=Y_{\rho'}$ if there is a mark $\rho'$ at $i+j$ such that $d(\rho,\rho')=1$, and $\eta_j=0$ otherwise.
    At the end, let $Y_0^*:=Y_{\rho_0}$. It follows that $Y_0^*$ is a measurable function of the update marks located at sites at distance less than $\lfloor l/2 \rfloor$ from the origin.

    Defining similarly $Y_i^*$ for all $i \in \Z^d$, it yields an $l-$dependent field. It remains to check that it coincides with $X^*$ with high probability.
    Let us prove that every update mark $\rho$ such that $d(\rho_0,\rho) \leq \lfloor l/2 \rfloor$ satisfies
    \begin{equation}\label{eq:coincidence_dry_marks}
        \{\rho \text{ is dry} \} \as\subset \{ Y_\rho=\varphi(\mathbf{0}, U_\rho)\}.
    \end{equation}
    For this, we first prove that if $\rho$ is dry, then every $\rho'$ between $\rho$ and $\C_l$ satisfies $Y_{\rho'} \in S_{d(\rho, \rho')}$. If $\rho' \in \C_l$, we have $Y_{\rho'} = \varphi(\mathbf{0},U_{\rho'})$; since $\rho$ is dry, we have that $U_{\rho'} \leq q_{d(\rho,\rho')}$ so $Y_{\rho'} \in S_{d(\rho,\rho')}$ by property \eqref{coupling_q_L} of the update function. Then, for $\rho' \notin \C_l$, we again argue recursively: assume that the property holds for every $\rho''$ between $\rho'$ and $\C_l$; we therefore have that all the update marks $\rho''$ such that $d(\rho', \rho'')=1$ satisfy
    $$Y_{\rho''}\in S_{d(\rho,\rho'')} \subset S_{d(\rho,\rho')+1},$$
    because $d(\rho,\rho'') \leq d(\rho,\rho')+1$. Then, we have $U_{\rho'} \leq q_{d(\rho,\rho')}$ because $\rho$ is dry, so we get $Y_{\rho'} \in S_{d(\rho,\rho')}$ by \eqref{coupling_q_L}.
    Finally, we have that $Y_{\rho'} \in S_1$ for every $\rho'$ such that $d(\rho,\rho')=1$. Now, $\rho$ being dry also implies that $U_\rho \leq q_0$, so by \eqref{maximal_coupling}, we deduce that $Y_\rho= \varphi(\mathbf{0}, U_\rho)$, so \eqref{eq:coincidence_dry_marks} holds.

    If $\rho_0$ itself is dry, we directly obtain $Y^*_0=Y_{\rho_0}=\varphi(\mathbf{0}, U_{\rho_0})=X_0^*$.
    More generally, if the first dry cutset $\C$ for $(0,0)$ contains only marks at distance less than $\lfloor l/2 \rfloor$ from $\rho_0$, it follows that every $\rho=(i,t,u) \in \C$ satisfies $Y_\rho = \varphi(\mathbf{0}, u)=X_t^{\mathbf{0}, \tau_0}(i)$, where $\tau_0$ is defined as in the beginning of the proof.
    Then, the updates encoded by marks between $\rho_0$ and $\C$ produce the same values
    in the dynamics $(X_t^{\mathbf{0}, \tau_0})_{t \geq \tau_0}$ and in the definition of the $Y_\rho$'s. In the end, we get $Y_0^*=Y_{\rho_0}=X_0^{\mathbf{0}, \tau_0}(0)=X_0^*$.

    Therefore, it holds that $Y_0^* =X^*_0$ unless there exist (at least) $\lfloor l/2 \rfloor$ consecutive wet marks on an active path starting from $(0,0)$. Thus, according to Lemma \ref{cutset}, there exists $\lambda >0$ such that
    $$\P(X_0^* \neq Y_0^*) \leq \P(\exists \text{ a wet path of length $\lfloor l/2\rfloor$ starting from } (0,0)) \leq e^{-\lambda l},$$
    
    Now, by translation-invariance, for every $\Lambda \Subset \Z^d$, it holds that
    $$\P(X^*|_\Lambda \neq Y^*|_\Lambda)=\P(\exists i\in \Lambda, X^*_i \neq Y_i^*) \leq |\Lambda| e^{-\lambda l},$$ hence the result.
\end{proof}

\section{Proofs for the Gaussian models}\label{Proofs for the Gaussian models}
In this Section, we prove the results stated in the introduction for our Gaussian models.
For this, we will use some properties on the Gaussian density that we recall here.
First, for all $x \in \R$, we have the following equality of distributions: $\Norm(x,1)=x+\Norm(0,1)$.
We will use the notation $\mathbb P (\Norm(x,1) \in \cdot)$ to denote the probability that a random variable of law $\Norm(x,1)$ belongs to some subset of $\R$.
Then, for all $L>0$, the function $x\mapsto \mathbb P(\Norm(x,1) \in [-L,L])$ defined over $\R$ is maximal for $x=0$, decreasing on $\R^+$ and it is an even function.
Moreover, we have the following bound on the Gaussian tail: for every $L>0$ and $\sigma>0$,
    \begin{align}\label{eq:gaussian_tails}
        \mathbb{P}(|\mathcal{N}(0,\sigma^2)| > L)&=\dfrac{2}{\sqrt{2\pi \sigma^2}} \int_L^{+\infty} e^{-t^2/(2\sigma^2)} \d t \notag \\ &\leq \dfrac{2}{\sqrt{2\pi \sigma^2}} \int_L^{+\infty} \dfrac{t}{L}e^{-t^2/(2\sigma^2)} \d t \notag \\ &\leq \dfrac{2 \sigma}{\sqrt{2\pi}L} e^{-L^2/(2\sigma^2)}.
    \end{align}

\subsection{Proof of Theorem \ref{thm_mu_L}}
Theorem \ref{thm_mu_L} will be deduced from Theorem \ref{thm_compact}, part 1 by proving that when $|\varepsilon|$ is small enough, $\mu_\varepsilon^L$ satisfies the high-noise assumption, and part 2 follows by \eqref{eq:distance_to_l-dep}, of which we give a proof below.

\begin{lemme}\label{lemma_eta_L}
One has $\gamma(\mu_\varepsilon^L) \xrightarrow[\varepsilon \to 0]{} 1$.
\end{lemme}

\begin{proof}
We deduce the result from Lipschitz continuity of $\overline{g_x}^L(t)$ with variable $x$.

Fix $t \in [-L,L]$.
First, it is easy to derive from mean value inequality that for all $x,y \in \R$, $|g_x(t) - g_y(t)| \leq c |x-y|$ with $c=\frac{e^{-1/2}}{\sqrt{2\pi}}$.
Set 
$$\lambda_\varepsilon = \inf_{|x| \leq L} \int_{-L}^L g_{\varepsilon x}(s) \d s = \inf_{|x| \leq L} \mathbb{P}(\Norm(\varepsilon x, 1) \in [-L,L])=\mathbb{P}(\Norm(\varepsilon L,1) \in [-L,L]).$$

Then, for all $x,y \in [L, L]$,
\begin{align*}
    |\overline{g_{\varepsilon x}}^L(t) - \overline{g_{\varepsilon y}}^L(t)| &= \left| \dfrac{g_{\varepsilon x}(t)}{\int_{-L}^L g_{\varepsilon x}(s) \d s} - \dfrac{g_{\varepsilon y} (t)}{\int_{-L}^L g_{\varepsilon y}(s) \d s} \right|\\
    &\leq \dfrac{|g_{\varepsilon x}(t) - g_{\varepsilon y} (t)|}{\int_{-L}^L g_{\varepsilon x} (s) \d s} + g_{\varepsilon y}(t) \left| \dfrac{1}{\int_{-L}^L g_{\varepsilon x}(s) \d s} - \dfrac{1}{\int_{-L}^L g_{\varepsilon y}(s) \d s}\right|\\
    &\leq \dfrac{c}{\lambda_\varepsilon} |\varepsilon x-\varepsilon y| + \dfrac{g_{\varepsilon y}(t)}{\lambda_\varepsilon^2} \left|\int_{-L}^L (g_{\varepsilon y}(s) - g_{\varepsilon x}(s)) \d s  \right| \\
    &\leq \dfrac{c|\varepsilon|}{\lambda_\varepsilon} |x-y| + \dfrac{2Lc}{\lambda_\varepsilon^2} |\varepsilon x- \varepsilon y| \leq \dfrac{3Lc|\varepsilon|}{ \lambda_\varepsilon^2} |x-y|,
\end{align*}
since $g_{\varepsilon y}(t) \leq 1$, $\lambda_\varepsilon^2 \leq \lambda_\varepsilon$, and assuming $L\geq 1$ (if not, just replace the $L$ by 1 in the last inequality).
Then, let $\eta(\varepsilon,L):=1-\gamma(\mu_\varepsilon^L)$ and use Lipschitz continuity to get the result:
\begin{align*}
    \eta(\varepsilon,L) &= 1- \int_{-L}^L \inf_{|x| \leq L } \overline{g_{\varepsilon x}}^L(s) \d s\\
    &= \int_{-L}^L \overline{g_0}^L(s) - \inf_{|x| \leq L } \overline{g_{\varepsilon x}}^L(s) \d s\\
    &= \int_{-L}^L \sup_{|x| \leq L} (\overline{g_0}^L(s) -\overline{g_{\varepsilon x}}^L(s) ) \d s\\
    &\leq 2L \dfrac{3Lc}{\lambda_\varepsilon^2} |\varepsilon| L = \dfrac{6L^3c}{\lambda_\varepsilon^2} |\varepsilon|.
\end{align*}

Finally, we just need to check that $\lambda_\varepsilon$ does not cancel at the limit $\varepsilon \to 0$ to conclude.
One can notice that $\lambda_\varepsilon = \mathbb P(\Norm(0,1) \in [(-1+\varepsilon) L,(1+\varepsilon)L]$ remains greater than $\mathbb P(\Norm(0,1) \in [0,2L])$ for $|\varepsilon| \leq 1$, hence the result.
\end{proof}

\begin{prop}\label{prop:radius_controls_distance_l-dep}
    Assume that $\mu$ is an FIID with coding radius $R$.
    Then, for every $l\in \N$, there exists an $l$--dependent distribution $\nu_l$ on $\Omega$ such that for all $\Lambda \Subset \Z^d$,
    $$d_{\mathrm{TV}}(\mu|_\Lambda, \nu_l |_\Lambda) \leq |\Lambda| \P(2R > l).$$
\end{prop}

\begin{proof}
    By assumption, there exists an i.i.d. process $Y=(Y_i)_{i \in \Z^d}$ and a factor map $F$ such that $F(Y)$ has law $\mu$.
    Let $l\in \N$ and $r=\lfloor l/2 \rfloor$.
    For every $i \in \Z^d$, let $Y^i$ be the truncation of $Y$ to the ball $B^i_r$, that is, $Y^i_j = \ind{j \in B^i_r} Y_j$ for all $j \in \Z^d$.
    Set $Z^l = (F(Y^i)_i)_{i \in \Z^d}$ and call $\nu_l$ the distribution of $Z^l$.
    Then, if two subsets $A$ and $B$ of $\Z^d$ are at distance greater than $l\geq 2r$, it follows that $Z^l|_A$ and $Z^l|_B$ are factors of independent random fields (which are respectively the restriction of $Y$ to $\bigcup_{i\in A}B^i_r$ and $\bigcup_{j\in B}B^j_r$, which are disjoint).
    Thus, $\nu_l$ is $l$--dependent.

    Let $\Lambda \Subset \Z^d$.
    For every $i \in \Z^d$, on the event $\{R_i(Y) \leq r\}$ we have $Z^l_i=F(Y)_i$.
    Therefore, on $\bigcap_{i \in \Lambda}\{R_i(Y) \leq r\}$, $Z^l|_\Lambda$ and $F(Y)|_\Lambda$ coincide.
    It yields a coupling of the measures $\mu|_\Lambda$ and $\nu_l |_\Lambda$ such that the probability that the two realisations coincide in the coupling is greater than
    $$\P\left(\bigcap_{i \in \Lambda}\{R_i(Y) \leq r\}\right) \geq 1- |\Lambda| \P(R>r),$$
    by a union bound.
    With the characterisation of the total variation distance as one minus the maximal coupling probability, it gives the result.
\end{proof}

\begin{proof}[Proof of Theorem \ref{thm_mu_L}]
    For $\mu=\mu_\varepsilon^L$, the local Markov property \eqref{Markov_property} holds with boundary set $B=B^0_1-\{0\}$, hence $|B| =2d$.
    According to Lemma \ref{lemma_eta_L}, $\gamma(\mu_\varepsilon^L)$ tends to $1$ when $\varepsilon$ goes to $0$, so that for $|\varepsilon|$ small enough, this quantity stays above $1-1/(2d)$.
    In this case, it follows from Theorem \ref{thm_compact} that $\mu_\varepsilon^L$ is an FFIID with exponential tails.
    This proves part 1 of Theorem \ref{thm_mu_L}, and part 2. follows immediately thanks to Prop \ref{prop:radius_controls_distance_l-dep}, since the coding radius $R$ of the factor map for $\mu_\varepsilon^L$ satisfies $\P(R>n) \leq e^{-cn}$ for some $c>0$ and all $n\geq 1$ (see the end of the proof of Theorem \ref{thm_compact}).
\end{proof}

\subsection{Proof of Theorem \ref{thm_mu}}
We prove Theorem \ref{thm_mu} as a special case of Theorem \ref{thm_general}. For this, the main assignment is to find suitable sequences $(S_n)_{n \geq 1}$ and $(q_n)_{n \geq 0}$ satisfying \eqref{link_L0_p0}, \eqref{link_Ln_pn}, \eqref{condition_qn} and \eqref{condition_Ln}.
We will search for $S_n$ of the form $[-L_n,L_n]$, where $(L_n)_{n\geq 1}$ is an increasing sequence of real numbers that goes to $+\infty$.

\begin{proof}[Proof of Theorem \ref{thm_mu}]
First, recall that $\mu_\varepsilon$ is a stationary Markov random field with boundary set $B=B_1-\{0\}$, so that $|B| = 2d$.
Then, for all $\eta \in S^B$, $\mu_\varepsilon( X_0 \in \cdot \mid X|_B=\eta)$ is Gaussian of mean $\frac{\varepsilon}{2d}\sum_{j \sim i} \eta_j$ and variance 1. Therefore,
$\tilde \gamma(\mu_\varepsilon)$ corresponds to the maximal probability for realisations of $\Norm(\varepsilon x, 1)$ for $x \in S_1=[- L_1, L_1]$ to coincide in a coupling and belong to $[-L_1,L_1]$, so we have
\begin{equation*}
    \tilde \gamma(\mu_\varepsilon) = \int_{-L_1}^{L_1} \inf_{|x| \leq L_1 } g_{\varepsilon x}(t) \d t.
\end{equation*}

To get \eqref{condition_qn}, one can choose $q_n = 1-a e^{-(2d)^n}$ with $a>0$ small enough.
Then, we want to find a sequence $(L_n)_{n\geq 1}$ that respects the links with $(q_n)_{n \geq 0}$ prescribed by \eqref{link_L0_p0} and \eqref{link_Ln_pn}, taking $S_n=[-L_n,L_n]$.
One can rewrite \eqref{link_Ln_pn} as 
$$\forall n \geq 1, \inf_{|x| \leq L_{n+1}} \mathbb{P}(| \mathcal{N}(\varepsilon x,1) | \leq L_n) \geq q_n.$$
The infimum is achieved for $x=\pm L_{n+1}$ so we need, for all $n \geq 1$, $$\mathbb P(|\mathcal{N}(\varepsilon L_{n+1},1)| > L_n) \leq 1-q_n = ae^{-(2d)^n}.$$
We have
\begin{align*}
    \mathbb{P}(|\mathcal{N}(\varepsilon L_{n+1},1)| > L_n) &=\mathbb P(\mathcal{N}(0,1) > L_n - \varepsilon L_{n+1})+\mathbb P(\mathcal{N}(0,1) < -L_n - \varepsilon L_{n+1}) \\&\leq 2\mathbb P(\mathcal{N}(0,1) > L_n - |\varepsilon| L_{n+1}).
\end{align*}
By the Gaussian tail bound \eqref{eq:gaussian_tails}, we get that this is lower than $ae^{-(2d)^n}$ if
\begin{equation}\label{croissance_Ln}
    L_n-|\varepsilon| L_{n+1} \geq \sqrt{2}\sqrt{(2d)^n - \log(a)}.
\end{equation}
Informally, we want $L_{n+1}$ bigger than $L_n$ but $|\varepsilon| L_{n+1}$ smaller that $L_n$, so we aim to take
$L_n \asymp \sqrt{|\varepsilon|} L_{n+1}$.
By choosing $L_n=L_1 |\varepsilon|^{-(n-1)/2}$ for some $L_1>0$, we have
$$L_n-|\varepsilon| L_{n+1} = L_1 \sqrt{|\varepsilon|}^{-(n-1)}(1-\sqrt{|\varepsilon}|),$$
which implies \eqref{croissance_Ln} so \eqref{link_Ln_pn} if $L_1$ is large enough and $|\varepsilon|$ small enough.

Repeating the last computation of the proof of Lemma \ref{lemma_eta_L}, observe that $\tilde \gamma(\mu_\varepsilon) \xrightarrow[\varepsilon \to 0]{}1$, so \eqref{link_L0_p0} holds for $|\varepsilon|$ small enough.

We finally check that the sequence $(L_n)_{n \geq 1}$ we defined satisfies \eqref{condition_Ln}, that is, we prove that
$$\mu_\varepsilon(\{\xi \in \Omega: \forall i \in \Z^d, |\xi_i| \leq L_{\max(n, \Vert i\Vert_1)} \}) \xrightarrow[n\to +\infty]{}1.$$

    Set for all $n\geq1$, $A_n=\{\xi \in \Omega: \forall i \in \Z^d, |\xi_i| \leq L_{\max(n, \Vert i \Vert_1)} \}$.
    When $(X_i)_{i \in \Z^d}$ is distributed according to $\mu_\varepsilon$, for all $i\in \Z^d$ the marginal law of $X_i$ is $\Norm(0, \sigma^2)$,
    with $\sigma^2 = \operatorname{Var}(X_i)$ (this does not depend on $i$ by stationarity). One can see, from Remark \ref{rem:def_gaussian_field}, that $$1 \leq \sigma^2 =  \Gamma(0,0) \leq 1+ \sum_{n \geq 1} \varepsilon^{2n} = \dfrac{1}{1-\varepsilon^2}.$$
    
    Then, by a union bound,
    \begin{align*}
        1-\mu_\varepsilon(A_n) &\leq \sum_{i \in B^0_n}  \mu_\varepsilon(|X_i| > L_n) + \sum_{i \notin B^0_n} \mu_\varepsilon (|X_i| > L_{\Vert i\Vert_1}) \\
        &\leq |B^0_n|\; \mathbb P (|\Norm(0, \sigma^2)|>L_n) + \sum_{k>n} |\partial B^0_k|\; \mathbb P (|\Norm(0, \sigma^2)|>L_k).
    \end{align*}
    The size of the balls and spheres of $\Z^d$ are of order given by $|B^0_n| \asymp n^d$ and $|\partial B^0_n| \asymp n^{d-1}$.
    Let $\tilde c>0$ be a constant such that $|B^0_n| \leq \tilde c n^d$ and $|\partial B^0_n| \leq \tilde c n^{d-1}$, and set $c= \dfrac{2 \sigma}{\sqrt{2\pi}} \tilde c$.

    Therefore, by the Gaussian tail bound \eqref{eq:gaussian_tails},
    $$1-\mu_\varepsilon(A_n) \leq c n^d \dfrac{e^{-L_n^2/(2\sigma^2)}}{L_n} + c \sum_{k>n} k^{d-1} \dfrac{e^{-L_k^2/(2\sigma^2)}}{L_k}.$$
    To get that this quantity tends to 0 as $n$ tends to $+ \infty$, we need the right-hand side series to converge and $n^d e^{-L_n^2/(2\sigma^2)}/L_n \to 0$.
    This is true when $n^{d-1} e^{-L_n^2/(2\sigma^2)}/L_n \leq \frac{1}{n^2}$ for sufficiently large $n$, which clearly holds for the sequence $L_n=L_1 |\varepsilon|^{-(n-1)/2}$.
\end{proof}

\appendix
\section{Appendix: construction of the update function}
In this Appendix, we present an explicit construction of the update function $\varphi$ of the dynamics, in the case where $S \subset \R$.
Returning to the general setting, $\mu$ is a stationary Markov random field distribution on $S^{\Z^d}$, with boundary set $B \Subset \Z^d \setminus\{0\}$. We consider the family of conditional distributions $\pi^\eta=\mu(X_0 \in \cdot \mid X|_B=\eta), \eta \in S^B$; $\gamma=\gamma(\mu)$ is the maximal coupling probability of distributions $\pi^\eta$ for $\eta \in S^B$.

For all $\eta \in S^B$, let $F(s \mid \eta):=\pi^\eta((-\infty,s])$ for all $s \in \R$ be the cumulative distribution function of $\pi^\eta$.
First, the classical (up to slight variations) construction of a maximal coupling of laws $\pi^\eta$ works as follows:
for all $s\in \R$, let $$R(s):=\Big(\bigwedge_\eta \pi^\eta\Big)((-\infty,s]) \et \widehat F(s \mid \eta):=F(s \mid \eta)-R(s)$$ (notice that $R(s) \xrightarrow[s\to +\infty]{}\gamma$ and $F(s \mid \eta) \geq R(s)$); then, set
$$\forall \eta \in S^B, \forall u\in[0,1], \varphi(\eta,u)=R^{-1}(u) \ind{u \leq \gamma}+\widehat F^{-1}(u-\gamma \mid \eta) \ind{u>\gamma},$$
where generically the notation $F^{-1}$ stands for the generalised inverse of a non-decreasing function $F$ (i.e. $F^{-1}(y)=\inf\{x : F(x)>y\})$.
It is clear that $\varphi$ satisfies \eqref{maximal_coupling_phi_L}, and easy to check that $\varphi(\eta,U)$ has distribution $\pi^\eta$ if $U \sim \mathcal{U}[0,1]$.

We now aim to construct a function $\varphi$ that satisfies \eqref{maximal_coupling} and \eqref{coupling_q_L}, which are the conditions imposed on the update function with respect to the levels. We suppose that the levels are of the form $S_n=[-L_n,L_n]$, for all $n \geq1$.
Recall that $\widetilde \gamma=\widetilde \gamma(\mu)$ is the maximal probability for random variables of laws $\pi^\eta, \eta \in S_1^B$ to coincide in a coupling, with the additional requirement that the common value belongs to $[-L_1,L_1]$.

For all $s \in \R$, let
$$\widetilde R(s):=\Big(\bigwedge_{\eta \in S_1^B} \pi^\eta\Big)((-L_1, s \wedge L_1]),$$
so that $\widetilde R(s)=0$ for $s \leq -L_1$ and $\widetilde R(s)=\tilde \gamma$ for $s\geq L_1$.
Let also for all $\eta \in S^B$ and all $a<b$
$$F(a,b \mid \eta):= \pi^\eta((-a,b]).$$
We introduce a special notation in the case $b=-a=L_n$: denote $q_n(\eta)=F(-L_n,L_n \mid \eta).$
Observe that, then, $q_n \leq \inf_{\eta \in S^B_{n+1}} q_n(\eta)$ (by \eqref{link_Ln_pn}).
Finally, for all $\eta \in S^B$ and $s \in \R$ let 
$$\widehat F_0(s \mid \eta):= \begin{cases}F(-L_1, s\wedge L_1\mid \eta)-\widetilde R(s) & \si \eta \in S_1^B \\ F(-L_1, s\wedge L_1\mid \eta) & \otw\end{cases}$$
and for all $n\geq 1$, let
$$\widehat F_n(s \mid \eta):= F(-L_{n+1}, s\wedge L_{n+1}\mid \eta)-F(-L_n,s\wedge L_n \mid \eta).$$
All these functions are non-decreasing so their generalised inverses are well-defined.

\begin{prop}
    For all $\eta \in S_1^B$ and $u\in [0,1]$, let 
    $$\varphi(\eta,u)= \widetilde R^{-1}(u) \ind{u \leq \tilde \gamma} + \sum_{n\geq 0} \widehat F_n^{-1}(u-q_n(\eta) \mid \eta) \ind{q_n(\eta) <u \leq q_{n+1}(\eta)}$$
    and for $\eta \in S^B \setminus S_1^B$, let 
    $$\varphi(\eta,u)= \sum_{n\geq 0} \widehat F_n^{-1}(u-q_n(\eta) \mid \eta) \ind{q_n(\eta) <u \leq q_{n+1}(\eta)}$$
    with convention $q_0(\eta):= \widetilde\gamma$.
    Then, if $U \sim \mathcal{U}[0,1]$, it follows that $\varphi(\eta,U)$ is distributed according to $\pi^\eta.$
    Moreover, for all $\eta \in S^B$ and $u\in[0,1]$, one has
    $$\eta \in S_1^B \et u \leq q_0 \implies \varphi(\eta,u)=\varphi(\mathbf{0},u) \in [-L_1,L_1]$$
    and for all $n \geq 1$,
    $$\eta \in S_{n+1}^B \et u \leq q_n \implies \varphi(\eta,u) \in [-L_n,L_n].$$
\end{prop}

\begin{proof}
    It is clear that the last two properties on $\varphi$, corresponding to \eqref{maximal_coupling} and \eqref{coupling_q_L}, are verified. This follows from \eqref{link_L0_p0} and \eqref{link_Ln_pn} and the fact that for all $n \geq 1$, $\widehat F_{n-1}^{-1}$ takes values in $[-L_n,L_n]$.
    Then, it remains to prove that $\varphi$ generates the right distribution.
    Let $U \sim \mathcal{U}[0,1]$, and denote by $\P$ the probability with respect to $U$.
    Let $s \in \R$ and $\eta \in S^B$. Suppose that $\eta \in S_1^B$ (the proof being similar if not). One has
    \begin{align*}
        \P( \varphi(\eta,U) \leq s) &= \P(\widetilde R^{-1}(U) \leq s \et U \leq \widetilde \gamma)+  \sum_{n\geq 0} \P\left( \widehat F_n^{-1}(U-q_n(\eta) \mid \eta) \leq s \et q_n(\eta) <U \leq q_{n+1}(\eta) \right) \\
        &=\P(U \leq \widetilde R(s)) + \sum_{n\geq0} \P\left(0<U-q_n(\eta) \leq \widehat F_n(s \mid \eta) \right) \\
        &=\widetilde R(s)+\sum_{n\geq0} \widehat F_n(s \mid \eta) \\
        &= F(-\infty,s \mid \eta),
    \end{align*}
    where in the second equality, we use the fact that $\widetilde R(s) \leq \widetilde\gamma$ and $\widehat F_n(s \mid \eta) \leq \widehat F_n(L_{n+1} \mid \eta)=q_{n+1} (\eta)-q_n(\eta).$
    This shows that $\varphi(\eta,U)$ has distribution $\pi^\eta$.
\end{proof}

\bibliographystyle{plain}
\bibliography{ref}

\end{document}